\documentclass[11pt,reqno]{amsart}
\usepackage{amsmath, amssymb, amsthm}
\usepackage{url}
\usepackage{tikz}
\usepackage[breaklinks]{hyperref}
\usepackage{graphics}
\usepackage{amsmath}
\usepackage{amssymb}
\usepackage{booktabs}
\usepackage{verbatim}
\usepackage{subfigure}

\usepackage{tikz}
\usetikzlibrary{decorations.markings}

\usepackage{nomencl}
\usepackage{afterpage}
\usepackage{emptypage}
\usepackage{amsthm}
\usepackage{array}
\usepackage{amsfonts}
\usepackage{upgreek}
\usepackage{mathtools}
\usepackage{pdfpages}
\usepackage{afterpage}
\usepackage{setspace}
\usepackage{graphicx}
\usepackage{fancyhdr}
\usepackage{amsthm}

\renewcommand{\Im}{\operatorname{Im}}

\renewcommand{\Re}{\operatorname{Re}}
\renewcommand{\Im}{\operatorname{Im}}

\renewcommand{\(}{\left\(}
\renewcommand{\)}{\right\)}
\renewcommand{\[}{\left\[}
\renewcommand{\]}{\right\]}

\numberwithin{equation}{section}
 \theoremstyle{plain}
\newtheorem{theorem}{Theorem}[section]
\newtheorem{lemma}[theorem]{Lemma}

\newtheorem{corollary}[theorem]{Corollary}

   \makeatletter
\def\proof{\@ifnextchar[{\@oproof}{\@nproof}}
\def\@oproof[#1][#2]{\trivlist\item[\hskip\labelsep\textit{#2 Proof of\
#1.}~]\ignorespaces}
\def\@nproof{\trivlist\item[\hskip\labelsep\textit{Proof.}~]\ignorespaces}

\makeatother

\begin{document}
\title[A Dirichlet character analogue of Ramanujan's formula for odd zeta values
 ]{A Dirichlet character analogue of Ramanujan's formula for odd zeta values}

\author{Anushree Gupta}
\address{Anushree Gupta\\ Department of Mathematics \\
The Pennsylvania State University \\
State College, Pennsylvania-16802,  USA.} 
\email{afg6050@psu.edu}

\author{Md Kashif Jamal}
\address{Kashif Jamal\\ Department of Mathematics \\
University of North Carolina at Charlotte \\
9201 University City Blvd, Charlotte, NC 28223,  USA.} 
\email{kashif.sxcr@gmail.com}

\author{Nilmoni Karak}
\address{Nilmoni Karak\\ Department of Mathematics \\
Indian Institute of Technology Kharagpur\\
Kharagpur-721302,  India.}
\email{nilmonikarak1999@gmail.com,  nilmonimath@kgpian.iitkgp.ac.in}

 \author{Bibekananda Maji}
\address{Bibekananda Maji\\ Department of Mathematics \\
Indian Institute of Technology Indore \\
Simrol,  Indore,  Madhya Pradesh 453552, India.} 
\email{bibek10iitb@gmail.com,  bmaji@iiti.ac.in}


\maketitle

\begin{abstract}
In 2001,  Kanemitsu, Tanigawa, and Yoshimoto studied the following generalized Lambert series, 
$$
 \sum_{n=1}^{\infty} \frac{n^{N-2h} }{\exp(n^N x)-1},
 $$ 
for  $N \in \mathbb{N}$ and $h\in \mathbb{Z}$ with some restriction on $h$. Recently, Dixit and the last author pointed out that this series has already been present in the Lost Notebook of Ramanujan with a more general form. Although, Ramanujan did not provide any transformation identity for it.  In the same paper,  Dixit and the last author found an elegant generalization of Ramanujan's celebrated identity for $\zeta(2m+1)$ while extending the results of Kanemitsu et al.  In a subsequent work,  Kanemitsu et al.  explored another extended version of the aforementioned series, namely, 
$$\sum_{r=1}^{q}\sum_{n=1}^{\infty} \frac{\chi(r)n^{N-2h}{\exp\left(-\frac{r}{q}n^N x\right)}}{1-\exp({-n^N x})},$$
where $\chi$ denotes a Dirichlet character modulo $q$,  $N\in 2\mathbb{N}$ and with some restriction on the variable $h$.  In the current paper,  we investigate the above series for {\it any} $N \in \mathbb{N}$ and $h \in \mathbb{Z}$.  We obtain a Dirichlet character analogue of Dixit and the last author's identity and there by derive a two variable generalization of  Ramanujan's identity for $\zeta(2m+1)$.  Moreover,  we establish a new  identity for $L(1/3,  \chi)$ analogous to Ramanujan's famous identity for $\zeta(1/2)$.  
  
\end{abstract}

\section{Introduction}
Evaluating an explicit formula for an infinite series has always been a challenging problem in mathematics.  Euler was the first who showed that all even zeta values can be written as
\begin{equation} \label{zeta(2m)}
 \zeta(2m)=	\sum_{n=1}^{\infty} \frac{1}{n^{2m}} =  r_{2m} \pi^{2m},   \quad  m\in \mathbb{N},  
\end{equation}
where $r_{2m}= (-1)^{m+1} \frac{2^{2m}B_{2m}}{2 (2m)!}$ and $B_{2m}$ denotes the $2m$th Bernoulli number.  From the above formula,  one can immediately conclude that $\zeta(2m)$ converges to a transcendental number.  
The next instinctive question arises about the arithmetic nature of odd zeta values.
  Is there any formula for $\zeta(2m+1)$ that is similar to the formula \eqref{zeta(2m)}? 
 The  nature of $\zeta(2m+1)$, $m\in\mathbb{N}$, remains mystery except $\zeta(3)$.  In 1979,  Roger Ap\'{e}ry \cite{Apery79, Apery 80} made a breakthrough by proving that the irrationality of $\zeta(3)$.   But, we are still unaware of the algebraic nature of $\zeta(3)$. In $2001$,  Ball  and Rioval \cite{ballrivoal} made an advancement by showing that there are
  infinitely many irrational odd zeta values.  Strikingly,  Zudilin \cite{zudilin} gave an impressive result around the same period.  
He showed  that ``at least one  among $\zeta(5), \zeta(7),\zeta(9)$ and $\zeta(11)$ is irrational".    

  Ramanujan in his Second Notebook \cite[p.~173,  Entry 21(i)]{Notebook volume 2} as well as in  Lost Notebook \cite[p.~320,  Formula (28)]{lnb},  noted down the following appealing identity for  $\zeta(2m+1)$.
  
{\it  Let $\alpha,  \beta \in \mathbb{R}^{+}$ with $\alpha \beta = \pi^2$.  For $m \in \mathbb{Z}-\{0\}$,  we have 
\begin{align}\label{Ramanujan's formula}	
F_{m}(\alpha) =  (-1)^{m} F_{m}(\beta) -2^{2m} \sum_{k=0}^{m+1} (-1)^k \frac{  B_{2k}B_{2m+2-2k}}{(2k)!(2m+2-2k)!} \alpha^{m+1-k} \beta^k,  
\end{align}}
where 
\begin{align}\label{Ramanujan function F_m(x)}
F_{m}(x) =
x^{-m}  \left( \frac{1}{2} \zeta(2m +1)  +  \sum_{n=1}^{\infty} \frac{n^{-2m-1}}{e^{2n x} -1 } \right).
\end{align}
Here we note that the finite sum involving Bernoulli numbers in \eqref{Ramanujan's formula} should be considered as zero sum when $m \leq -2$.   In 1977,   Berndt \cite{Berndt77} showed that Euler's formula \eqref{zeta(2m)} and Ramanujan's formula \eqref{Ramanujan's formula} can be derived from a modular transformation property for a generalized Eisenstien series.  Various generalizations of \eqref{Ramanujan's formula} have been discovered by mathematicians over the years,  readers are encouraged to see \cite{BGK23,   Berndt77,  BS16,  BS17,  bradley,  CCVW21, CJM22,  DGKM20,  DM20, Gross72, GMR,  GM21}.  In 1972,  Grosswald \cite{Gross72} found a generalization of \eqref{Ramanujan's formula} in the upper half plane from which one can obtain transformation formula for Eisentien series over the full modular group.  
The following Lambert series,  $m \in \mathbb{Z}$,  $x>0$,   
$$\sum_{n=1}^{\infty} \frac{n^{-2m-1}}{e^{2nx} -1 }$$ present in \eqref{Ramanujan function F_m(x)}  has got wide significance due to its connection with the Fourier series expansion of the Eisenstien series.   In 2001, Kanemitsu,  Tanigawa and Yoshimoto  \cite{KTY01}  investigated its generalized form defined by
\begin{equation}\label{Kanemitsu lambert series}
	\sum_{n=1}^{\infty} \frac{n^{N-2h}}{\exp(n^N x) - 1},
\end{equation} 
where $N\in\mathbb{N}$ and $h$ is an integer lying in the interval  $[1,  N/2]$.
They were able to derive interesting identities for $\zeta(s)$ at rational arguments while studying the above infinite series.  For example,  they obtained Ramanujan's identity for $\zeta(1/2)$.  
Recently,  Dixit and the last author \cite{DM20} noticed that the series \eqref{Kanemitsu lambert series} is in fact present in the Lost Notebook \cite[p.~332]{lnb} of Ramanujan with a more general form.  However, Ramanujan did not provide any transformation identity for it.  This motivated Dixit and Maji \cite{DM20} to study  \eqref{Kanemitsu lambert series} further.  Quite unexpectedly,  they found a beautiful generalization of Ramanujan's celebrated identity for $\zeta(2m+1)$ while extending results of Kanemitsu et al.  Mainly,  Dixit and Maji \cite[Theorem 1.2]{DM20} obtained the following generalized identity which connects two distinct odd zeta values: 

{\it Suppose $N$ is an odd natural number and  $\alpha,  \beta \in \mathbb{R}^{+}$ with  $\alpha\beta^{N}=\pi^{N+1}$.  Then for $m \in \mathbb{Z}-\{0\}$,   we have
\begin{align}\label{zetageneqn}
	&\alpha^{-\frac{2Nm}{N+1}}\left(\frac{1}{2}\zeta(2Nm+1)+\sum_{n=1}^{\infty}\frac{n^{-2Nm-1}}{\textup{exp}\left((2n)^{N}\alpha\right)-1}\right)   =\left(-\beta^{\frac{2N}{N+1}}\right)^{-m}\frac{2^{2m(N-1)}}{N}  \nonumber\\
	& \times \Bigg(\frac{1}{2}\zeta(2m+1)+(-1)^{\frac{N+3}{2}}\sum_{j=-\frac{(N-1)}{2}}^{\frac{N-1}{2}}(-1)^{j}\sum_{n=1}^{\infty}\frac{n^{-2m-1}}{\textup{exp}\left((2n)^{\frac{1}{N}}\beta e^{\frac{i\pi j}{N}}\right)-1}\Bigg)\nonumber\\
	&+(-1)^{m+\frac{N+3}{2}}2^{2Nm}\sum_{j=0}^{\left\lfloor\frac{N+1}{2N}+m\right\rfloor}\frac{(-1)^jB_{2j}B_{N+1+2N(m-j)}}{(2j)!(N+1+2N(m-j))!}\alpha^{\frac{2j}{N+1}}\beta^{N+\frac{2N^2(m-j)}{N+1}}.
\end{align}}
Here we emphasize that the following integral representation of \eqref{Kanemitsu lambert series} was the starting point of the work of Dixit and Maji:  
\begin{align} \label{Dixit -Maji integral}
	\sum_{m=1}^{\infty} \frac{m^{N-2h}}{\exp(m^N x) - 1} =  \frac{1}{2\pi i} \int_{c_0-i \infty}^{c_0+ i \infty} \Gamma(s) \zeta(s) \zeta(Ns-(N-2h)) x^{-s} \mathrm{d}s,
\end{align}
where $c_0$ is some large positive quantity.  Motivated from this representation,  recently,  Gupta and the last author \cite{GM21} also studied a similar integral to find an extension of Ramanujan's identity \eqref{Ramanujan's formula} in a different direction.   In a subsequent paper,   Dixit et al.  \cite{DGKM20} also studied the above integral \eqref{Dixit -Maji integral} associated to the Hurwitz zeta function $\zeta(s,a)$,  inspired from another work of Kanemitsu et al.  \cite{Kanemitsu Hurwitz}.  
Mainly,  they investigated the following infinite series and its integral representation: 
\begin{align}\label{Hurwitz_integral}
\sum_{m=1}^{\infty} \frac{m^{N-2h} \exp\left(-a\, m^N x\right)}{\exp(m^N x) - 1}   =  \frac{1}{2\pi i} \int_{c_0-i \infty}^{c_0+ i \infty} \Gamma(s) \zeta(s,  a) \zeta(Ns-(N-2h)) x^{-s} \mathrm{d}s.  
\end{align}
While studying \eqref{Hurwitz_integral},  Dixit et al.  \cite[Theorem 2.4]{DGKM20} obtained a gigantic two-variable extension of Ramanujan's identity \eqref{Ramanujan's formula} by which they were able to connect many odd zeta values from a single identity.  
A few years later,   Kanemitsu et al.  \cite{KTY L-function} further explored a character analogue of the series \eqref{Kanemitsu lambert series},  namely,  the following infinite series and its integral representation:
\begin{equation}\label{Kanemitsu_Dirichlet L-function}
	\sum_{r=1}^{q}\sum_{n=1}^{\infty}\frac{\chi(r)n^{N-2h}\exp{\left(-\frac{r}{q}n^N x\right)}}{1-\exp({-n^N x})} =\frac{1}{2\pi i} \int_{c_0-i \infty}^{c_0+ i \infty}  \Gamma(s) L(s,  \chi) \zeta(Ns -N +2h)  \left(\frac{x}{q}\right)^{-s} \mathrm{d}s,
\end{equation}
 where $\chi$ is a Dirichlet character modulo $q$,  and for some large positive $c_0$.  
  For any $N \in 2 \mathbb{N}$ and $h \in \mathbb{Z}$ with some restriction on $h$,   they were able to obtain many interesting identities for  the Dirichlet L-function $L(s,\chi)$ at different rational arguments from a transformation formula for the infinite series \eqref{Kanemitsu_Dirichlet L-function}.  For example,  they obtained formulae for $L(1/2,  \chi)$ and $L(1/4,  \chi)$ analogous to Ramanujan's famous identity for $\zeta(1/2)$.  
   
   In this paper, we study the infinite series \eqref{Kanemitsu_Dirichlet L-function} for \emph{any} $N\in \mathbb{N}$ and $h \in \mathbb{Z}$ without any restriction on $h$.  Interestingly,  we obtain a new Dirichlet character analogue of  Ramanujan's identity \eqref{Ramanujan's formula} and as a consequence we derive an identity for $L(1/3,  \chi)$.   

Let $q\geq 1$ be an integer and $\chi$ be a Dirichlet character modulo $q$.  We define a generalized divisor function
	\begin{align}\label{Definition of new variant}
\sigma_{r,\chi}^{(N)}(n) := \sum_{d^{N}|n} d^r \chi\left(\frac{n}{d^N}\right),
\end{align}
where $r \in \mathbb{C}$ and $ N \in \mathbb{N}$.  When $q=1$,  the above divisor function reduces to  
 \begin{align*}
\sigma_{r}^{(N)}(n) := \sum_{d^{N}|n} d^r.
\end{align*}
Recently,  a Voronoi summation formula associated to $\sigma_{r}^{(N)}(n)$  has been studied by Dixit,  Maji,  and Vatwani  \cite{DMV23}.  To know more about transformation formula for Lambert series associated to  $\sigma_{r}^{(N)}(n)$,   readers are encouraged to see 
\cite{ BDG23,  BM23}.  For $\Re(s)> \max\{1,  \frac{1+r}{N} \}$,  one can check that the Dirichlet series associated to the divisor function $ \sigma_{r,\chi}^{(N)}(n)$ is the following product: 
\begin{align}\label{series}
\sum_{n=1}^\infty  \frac{\sigma_{r,\chi}^{(N)}(n)}{n^s} = L(s,  \chi) \zeta(N s - r).  
\end{align}



Now we are ready to state main results of our paper.  
\section{Main results}

	\begin{theorem}\label{Main theorem}
	Let $x \in \mathbb{R}^{+},  q \in \mathbb{N}$, and  $\chi$ be a primitive character modulo $q$.   Let $N\in \mathbb{N}$ and $h\in\mathbb{Z}$ with $N-2h \neq -1$. 
Let us define 
\begin{equation}\label{F-function}
F(2h-N,  x,  \chi):= \sum_{r=1}^{q}\sum_{n=1}^{\infty}\frac{\chi(r)n^{N-2h}\exp{\left(-\frac{r}{q}n^N x \right)}}{1-\exp({-n^N x})}, 
\end{equation}
and 
\begin{equation}\label{G-function}
 G_{j}\left( \frac{ 2h-1}{N}, x,  \chi \right) :=   \sum_{n=1}^{\infty}\frac{1}{n^{\frac{2h-1}{N}}} \frac{\chi(n)}{\exp \left( 2\pi \left( n x \right)^{1/N}\exp\left(-\frac{i\pi j}{2N}\right)\right)-1}. 
 \end{equation}
	Then,  we have
	\begin{align} \label{Main expression}
		F(2h-N,  x,  \chi)
	&	=  \zeta(-N+2h) L(0,\chi) + \mathcal{R}_1(x) \nonumber\\
	 &  + \frac{1}{N}\Gamma\left(\frac{N-2h +1}{N}\right)  L\left(\frac{N-2h+1}{N},\chi\right)\left(\frac{x}{q}\right)^\frac{2h-N-1}{N} \nonumber\\
	& +\sum_{j=1}^{2\lfloor\frac{h}{N}\rfloor + 1} \frac{(-1)^{j+1}}{(j+1)!} B_{{j+1},\chi} \  \zeta(-Nj -N +2h)\left(\frac{x}{q}\right)^{j} \nonumber\\& + \mathcal{J}_\chi(x),
	\end{align} 
where \begin{equation} \label{R(x)}
	\mathcal{R}_1(x):=\begin{cases}
		\frac{\zeta(2h)}{x},  & \text { if } q =1,  \\
	0,  & \text { if } q> 1, 
	\end{cases}\\
\end{equation}
and
\begin{align}\label{J}
	\mathcal{J}_\chi(x) := (-1)^{h+1}  \frac{\mathcal{G}(\chi)}{ N}  \left(\frac{2\pi}{x}\right)^\frac{N-2h+1}{N} \sideset{}{''}\sum_{j=-(N-1)}^{N-1} &  v_{N,\chi}(j) \exp\left(\frac{i\pi j(2h-1)}{2N}\right) \nonumber \\
	 &\times G_{j}\left( \frac{ 2h-1}{N},  \frac{2\pi}{x},  \bar{\chi} \right),
\end{align}
 where $''$ indicates that the sum over $j$ takes the values $ j=-(N-1),-(N-3), \cdots ,(N-3),(N-1)$ 
and \begin{equation} \label{vv}
	v_{N,\chi}(j)=\begin{cases}
		1,  & \text { if } \chi \text { is even, }  N \in \mathbb{N}, \\
	(-1)^\frac{N}{2} i^{j-1},  & \text { if } \chi \text { is odd, } N \in 2\mathbb{N}.
		\end{cases}
\end{equation}
\end{theorem}


	
%
%

\subsection{A Character analogue of Ramanujan's formula for $\zeta(2m+1)$}


As an application of the above theorem,  we obtain the following identity,
which connects $L(2m+1,  {{\chi}})$ and $\zeta(2Nm+1)$.

\begin{theorem}\label{character analogue_GJKM}
Let $N\geq 1$ be an odd integer.  Let $\chi$ be an even character modulo $q$ and $\mathcal{G}(\chi)$ be the Gauss sum defined as in \eqref{Gauss sum}.  If $\alpha$ and $\beta$ be two positive real numbers such that $\alpha\beta^{N}=\pi^{N+1}$, then for any $m\in\mathbb{Z}\backslash \{0\} $, we have

\begin{align} \label{Main expression for N greater than 1}
&	 \alpha^{-\frac{2Nm}{N+1}} \left\{ \sum_{r=1}^{q}\sum_{n=1}^{\infty} \frac{\chi(r) \exp{\left(-\frac{r}{q}(2n)^{N} \alpha\right)}}{ n^{2Nm+1} \left( 1-\exp\left({-(2n)^{N} \alpha} \right)\right)} - \zeta(2Nm+1) L(0,\chi) \right\} \nonumber \\
 & = \frac{(-1)^m \beta^{-\frac{2Nm}{N+1}} 2^{2m(N-1)}}{N} \Bigg\{ \frac{q}{2 \mathcal{ G(\overline{\chi})} }  L\left(2m+1,\bar{\chi}\right) +(-1)^{{\frac{N+3}{2}}}\mathcal{G}(\chi) \sum_{j=- \frac{(N-1)}{2}}^{\frac{N-1}{2}} (-1)^j   \nonumber \\
 & \hspace{5cm} \times 
	  \sum_{n=1}^{\infty} \frac{1}{n^{2m+1}} \frac{\bar{\chi}(n)}{\exp \left( (2n)^{1/N} \beta \exp\left(-\frac{i\pi j}{N}\right)\right)-1} \Bigg\} \nonumber\\
	 &+ (-1)^{m+ \frac{N+3}{2}} 2^{2Nm}  \sum_{j=1}^{\lfloor m+ \frac{N+1}{2N} \rfloor} \frac{(-1)^{j}}{q^{2j-1}} \frac{B_{{2j},\chi}}{ (2j)! } \frac{B_{2N(m-j)+N+1}}{(2N(m-j)+N+1 )!} \alpha^{\frac{2j}{N+1}} \beta^{N+\frac{2N^2(m-j)}{N+1}} + \mathcal{R},  \nonumber
	\end{align} 
where \begin{equation} 
	\mathcal{R}:=\begin{cases}
	(-1)^{ m + \frac{N+3}{2}}  2^{2Nm} \frac{B_{2Nm+N+1}}{(2Nm+N+1)!} \beta^{N+\frac{2N^2 m}{N+1}}	,  & \text { if } q =1, \\
	0, & \text { if } q> 1. 
	\end{cases}\\
\end{equation}

\end{theorem}
	Letting $q=1$ i.e.,  $\chi$ being the trivial character in Theorem \ref{character analogue_GJKM} and upon simplification,  one can immediately recover the identity \eqref{zetageneqn} of Dixit and Maji.  
Moreover,  substituting $N=1$ and $q=1$ in Theorem \ref{character analogue_GJKM},  we can derive Ramanujan's identity \eqref{Ramanujan's formula}.  Thus,  the above theorem can be considered as a two variable generalization of Ramanujan's formula \eqref{Ramanujan's formula}. 
	
Now considering $q>1$ and letting $N=1$ in Theorem \ref{character analogue_GJKM},  we obtain the following identity of Katayama \cite[p.~235]{Katayama74}. 
\begin{corollary}\label{Character analogue for q>1}
Let $q>1$ and  $\chi$ be a primitive even character modular $q$.   Let $\alpha$ and $\beta$ be two positive real numbers such that $\alpha \beta = \pi^2$,  then for any non-zero integer $m$, we have
{\allowdisplaybreaks
\begin{align}
& \alpha^{-m} \left\{ \sum_{r=1}^{q}\sum_{n=1}^{\infty} \frac{\chi(r) \exp{\left(-\frac{r}{q}2n \alpha\right)}}{n^{2m+1} \left( 1-\exp\left({-2n \alpha}\right) \right)}  \right\} \nonumber \\
 & = (- \beta)^{-m}  \Bigg\{ \frac{q}{2 \mathcal{ G(\overline{\chi})} }  L\left(2m+1,\bar{\chi}\right) +\mathcal{G}(\chi) \sum_{n=1}^{\infty} \frac{1}{n^{2m+1}} \frac{\bar{\chi}(n)}{\exp \left( 2n \beta\right)-1} \Bigg\} \nonumber\\
	 &+ (-1)^{m} 2^{2m}  \sum_{j=1}^{ m+ 1} \frac{(-1)^{j}}{q^{2j-1}} \frac{B_{{2j},\chi}}{ (2j)! } \frac{B_{2m-2j+2}}{(2m-2j+2)!} \alpha^{j} \beta^{m-j+1}.   \label{character_Katayama}
\end{align}}
\end{corollary}

\subsection{A new identity for $L(1/3,  \chi)$.}
As an application of our main Theorem \ref{Main theorem}, 	we obtain an interesting identity for $L(1/3,  \chi)$. 
	
\begin{corollary}\label{L(1/3,chi)}
Let $\chi_5$ be the even primitive character modulo $5$.  Let $\alpha$ and $\beta$ be positive real numbers such that $\alpha^3 \beta = 2 \pi^4$.  Then we have
\begin{align*}
	\sum_{r=1}^{5} \sum_{n=1}^{\infty} \chi_5(r) \frac{n \exp\left(-\frac{r}{5} (n\alpha)^3 \right)}{1- \exp\left(-(n\alpha)^3 \right)}&=\frac{\mathcal{G}(\chi_5)}{3}  \left(\frac{2\pi}{\alpha^3}\right)^{\frac{2}{3}}\Bigg\{ L\left(\frac{1}{3},\chi_5\right) \\
	&+ \sum_{m=1}^{\infty} \frac{{\chi_5}(m)}{m^\frac{1}{3}}
 \left(\frac{e^{- \sqrt[3]{m \beta}}}{2\sinh(\sqrt[3]{m \beta})} + \frac{\cos\left(\sqrt{3} \sqrt[3]{m \beta} +\frac{\pi}{3}\right)-\frac{e^{-\sqrt[3]{m \beta}}}{2}}{\cosh(\sqrt[3]{m \beta})-\cos\left(\sqrt{3} \sqrt[3]{m \beta}\right)}\right)\Bigg\}. 
\end{align*}

\end{corollary}
The above formula for $L(1/3, \chi_5)$ is analogous to  Ramanujan's famous formula \cite[p.~332]{lnb}  for $\zeta(1/2)$. 
{\it 
Let $\alpha$ and $\beta$ be two positive numbers such that $\alpha \beta=4\pi^{3}$.Then
\begin{align}\label{Ramanujan_zeta(1/2)}
\sum_{n=1}^{\infty}\frac{1}{e^{n^{2}\alpha}-1}=\frac{1}{4}+\frac{\pi^2}{6 \alpha}+\frac{\sqrt{\beta}}{4\pi}\zeta\left(\frac{1}{2}\right)+\frac{\sqrt{\beta}}{4\pi}\sum_{m=1}^{\infty}   \frac{\cos(\sqrt{m\beta})-\sin(\sqrt{m\beta})-e^{-\sqrt{m\beta}}} { \sqrt{m} ( \cosh(\sqrt{m\beta})-\cos(\sqrt{m\beta}))}.
\end{align}}
	
Now we recall that our Theorem \ref{Main theorem} is valid only for those $N$ and $h$ such that $N-2h \neq -1$.  Thus,  we obtain a separate identity for $N-2h = -1$.


\begin{theorem} \label{N-2h=-1 case}
	Let $\chi$ be an even primitive Dirichlet character modulo $q$, $q\geq1$ and $x>0$ be any real number. Then for any odd positive integer $N$,  we have
	\begin{align}
		\sum_{r=1}^{q}  \sum_{n=1}^{\infty}\frac{\chi(r)}{n}\frac{\exp{(-\frac{r}{q}n^N x)}}{1-\exp({-n^N x})} = \mathcal{R}_0 + \mathcal{R}_1 +   \mathcal{R}(N) + \mathcal{K}_{\chi,N}(x),
	\end{align}
	where 
	{\allowdisplaybreaks
	\begin{equation}\label{Residue at 0}
		\mathcal{R}_0= 
		\begin{cases}
			\frac{1}{2N} \left(\gamma(1-N) -\log (2\pi) +\log(x)\right),  & \text{if} \ \chi \ \text{is even and} \  q=1,\\
			-\frac{1}{2N} \sum_{r=1}^{q-1} \chi(r) \log\left(
			\sin\left(\frac{r\pi}{q}\right)\right), & \text{if} \ \chi \ \text{is even and}  \ q>1,\\
		\end{cases} \\
	\end{equation}}

\begin{equation} 
	\mathcal{R}_1 = \begin{cases}
		\frac{\zeta(N+1)}{x}, & \text{if} \ q=1, \\
		0, & \text{if} \ q>1, \\
	\end{cases}
\end{equation}
\begin{equation} 
	\mathcal{R}(N) = 
	\begin{cases}
		\frac{x}{2q} L(-1,\chi), & \text{if} \ N=1, \\
		0, & \text{if} \ N>1, 
	\end{cases}
\end{equation}
and
\begin{align*} 
	\mathcal{K}_{\chi,N}(x) = (-1)^{\frac{N+3}{2}} \frac{\mathcal{G}(\chi)}{N} \sum_{n=1}^{\infty} \frac{\bar{\chi}(n)}{n}  \sideset{}{''}\sum_{j=-(N-1)}^{N-1}   \frac{\exp\left(\frac{i \pi j}{2}\right)}{\exp\left(2 \pi \left(\frac{2\pi n}{x}\right)^{(1/N)} \exp\left(\frac{-i \pi j}{2N}\right)\right)-1}. 
\end{align*}
\end{theorem}
In particular letting $q=1$ in the above theorem,  we obtain the below identity. 
\begin{corollary}\label{for q=1}
For any odd positive integer $N>1$,  one has
\begin{align*}
\sum_{n=1}^{\infty} \frac{1}{n\left( \exp(n^N x) -1 \right)} & =  \frac{1}{2N} \left(\gamma(1-N) +\log\left(\frac{x}{2\pi} \right) \right)  +  \frac{\zeta(N+1)}{x} \\
& + \frac{(-1)^{\frac{N+3}{2}}  }{N} \sideset{}{''}\sum_{j=-(N-1)}^{N-1}  \sum_{n=1}^{\infty} \frac{1}{n} \frac{\exp\left(\frac{i \pi j}{2}\right)}{\exp\left(2 \pi \left(\frac{2\pi n}{x}\right)^{(1/N)} \exp\left(\frac{-i \pi j}{2N}\right)\right)-1}. 
\end{align*}
\end{corollary}
Under suitable substitution,  we can show that the above identity is same as the identity obtained in \cite[Theorem 1.3]{DM20}.  Furthermore,  letting $q=N=1$ in Theorem \ref{N-2h=-1 case},  we obtain the following identity. 
\begin{corollary}\label{for N=q=1}
For any positive real $x$,  we have
\begin{align}
\sum_{n=1}^\infty \frac{1}{n\left(\exp(n x) -1  \right)} = \frac{1}{2} \log\left( \frac{x}{2\pi} \right) + \frac{\pi^2}{6 x} - \frac{x}{24} +  \sum_{n=1}^\infty \frac{1}{n\left(\exp\left(4 \pi^2 \frac{n}{ x} \right) -1  \right)}. 
\end{align}
\end{corollary}
Upon changing the variables $x= 2\alpha$ and $\alpha \beta =\pi^2$,  one can easily show that the above identity is exactly same as the following Ramanujan's identity \cite[Ch.~14,  Sec. ~8, Cor. (ii)]{Notebook volume 2}:
\begin{align}
\sum_{n=1}^\infty \frac{1}{n\left(\exp(2 n \alpha ) -1  \right)} - \sum_{n=1}^\infty  \frac{1}{n\left(\exp(2 n \beta ) -1  \right)}  = \frac{\beta -\alpha}{12} + \frac{1}{4} \log\left( \frac{\alpha}{\beta}  \right).
\end{align}
Interestingly,  this identity is equivalent to the transformation formula $\eta\left( - \frac{1}{z}\right) = \sqrt{-iz}\eta(z)$ of the Dedekind eta function $\eta(z)$.

	\section{Some well-known results}
	In this section,  we jot down a few well-known results which will be useful throughout the article. 
For $\Re(s)>0$,  the classical Gamma function $\Gamma(s)$ is defined by
\begin{equation*}
	\Gamma(s) := \int_{0}^{\infty} x^{s-1} \exp(-x) \mathrm{d}x.
	\end{equation*}
%
For $\Re(z) >0$ and $c>0$, we have
\begin{equation} 
\exp(-z) = \frac{1}{2\pi i} \int_{c-i \infty}^{c+i\infty} \Gamma(s) z^{-s}\mathrm{d}s  .
	\label{Inverse Mellin}
\end{equation}
Riemann established the analytic continuation of $\zeta(s)$  in $\mathbb{C}$ except at $s=1$, and it satisfies the following functional equation:
	\begin{equation}\label{asymetric functional equation for zeta}
		\zeta(s)= 2^s \pi^{s-1}\Gamma(1-s) \sin\left(\frac{\pi s }{2}\right)\zeta(1-s).
	\end{equation} 
We know that $\Gamma(s)$ has simple poles at non-positive integers with residue  at $-m$ is given by
\begin{equation} \label{Res_Gamaa}
\textrm{Res}_{s=-m} \Gamma(s) = \frac{(-1)^m}{m!}.
\end{equation}
Moreover,  $\Gamma(s)$ satisfies the following asymptotic expansion at $s=0$, 
	\begin{equation}	\label{laurent of Gamma}
		\Gamma(s) = \frac{1}{s} -\gamma + \frac{1}{2} \left(\gamma^2 + \frac{\pi^2}{6}\right)s - \frac{1}{6}\left(\gamma^3 +\frac{\gamma \pi^2}{2} + 2\zeta(3)\right)s^2 +O(s^3).
	\end{equation}
	For $s\in \mathbb{C}\backslash\mathbb{Z},$ $\Gamma(s)$ satisfies
		\begin{equation} \label{Euler's reflection}
	\Gamma(s)	\Gamma(1-s)	 = \frac{\pi}{\sin(\pi s) } .
	\end{equation}
Let $s = \sigma +iR$ and $p \leq \sigma \leq q$.  Stirling's bound for $\Gamma(s)$ \cite[p.~151]{IK} is given by
	\begin{equation}\label{stirling equn}
		\Gamma (\sigma + i R) \ll  | R|^{\sigma-\frac{1}{2}} \exp\left({- \frac{ \pi |R|}{2}}\right),
	\end{equation}
	as $|R|\rightarrow \infty$. 
The Gauss sum corresponding to a Dirichlet character $\chi$ modulo $q$ is defined as 
	\begin{equation} \label{Gauss sum}
		\mathcal{G}(\chi) := \sum_{r=1}^{q} \chi(r) e^{2\pi i r/q}.
	\end{equation}
Let us define
\begin{equation} \label{a}
	a:= \frac{1-\chi(-1)}{2} =\begin{cases}
		0, & \text { if } \ \chi \hspace{0.2cm} \text{is even}, \\
		1, &  \text { if} \ \ \chi \hspace{0.2cm}\text{is odd}. \\
	\end{cases}
\end{equation}
Now,  we assume that $\chi$ is any primitive  character modulo $q$. Then for every $s\in\mathbb{C}$, 
	\begin{equation} \label{Dirichlet L-function}
		L(s,\chi) = \varepsilon_\chi 2^s \pi^{s-1} q^{\frac{1}{2} -s}\sin\left(\frac{\pi}{2}(s+a)\right)\Gamma(1-s) L(1-s,\bar{\chi}),
	\end{equation}
	where $ \varepsilon_\chi = \frac{\mathcal{G}(\chi)}{i^a\sqrt{q}}$
	is an algebraic number of absolute value $1$.

	\begin{lemma}
		For any $\ell \in \mathbb{N}\cup \{0\}$, 
		$$\zeta(-\ell) = (-1)^\ell\frac{B_{\ell+1} }{\ell+1}.$$
		This implies that $\zeta(s)$ vanishes at $-2\ell$ as we know $B_{2\ell +1} =0$, for all  $\ell \in \mathbb{N}$. These zeros are known as  trivial zeros of $\zeta(s)$.
	\end{lemma}

The below lemma tells about trivial zeros of the Dirichlet $L$-function $L(s,  \chi)$.  
\begin{lemma}\label{Zeros_Dirichlet_L}
	Assume that $\chi$ is a primitive  character modulo $q$, with $q>1$. When $\Re(s)>1$, there are no zeros of $L(s,\chi)$ and
 for the case $\Re(s)\leq0$, there are zeros at certain negative integers.\\
	(i) If $\chi$ is an even primitive character, the only zeros of $L(s,\chi)$ are simple zeros at $0,-2,-4,-6,\cdots$\\
(ii) If $\chi$ is an odd primitive character, the only zeros of $L(s,\chi)$ are simple zeros at $-1,-3,-5,-7\cdots.$
\end{lemma}

The next two lemmas give information about the special values of $L(s,  \chi)$ at non-positive integers. 
	\begin{lemma}\cite[p.~186]{Cohen-Henri}.
	Let $\chi$ be a primitive  character modulo $q$. Then for every integers $k\geq 1$, we have 
\begin{equation}	
	 L(-k, \chi) = -\frac{B_{k+1,\chi
	}}{k+1}, \label{L-value}
	\end{equation}
\end{lemma}
where $B_{k,  \chi}$ denotes as the generalized Bernoulli number defined as 
	\begin{equation}
		\sum_{k=1}^{q} \chi(k)\frac{te^{kt}}{e^{qt}-1}=\sum_{n=0}^{\infty} B_{n,\chi} \frac{t^n}{n!}. \label{Generalized Bernoulli}
	\end{equation}

	
	\begin{lemma}\label{L(0,chi)} \cite[p.~188]{Cohen-Henri}
	Suppose $\chi$ is a Dirichlet character modulo $q$,  then we have
	\begin{equation*}
		L(0,\chi) =
		\begin{cases}
			0, & \text{if}\ \chi\ \text{is even and } \ q>1,\\
			-\frac{1}{2},  & \text{if} \ \ q=1,\\
			-\frac{1}{q} \sum_{r=1}^{q-1}\chi(r) r,  &\text{if}\ \chi\ \text{is odd.}
		\end{cases}
	\end{equation*}
\end{lemma}
\begin{lemma} \label{L'(0,chi)}\cite[p.~189]{Cohen-Henri}
	Suppose $\chi$ is a character modulo $q.$ Then
	\begin{equation*}
		L^{\prime}(0,\chi) =
		\begin{cases}
			-\frac{1}{2} \sum_{r=1}^{q-1} \chi(r) \log\left(
			\sin\left(\frac{r\pi}{q}\right)\right) &  \text{if} \ \chi \ is \  \text{even and non-principle,}\\
			\sum_{r=1}^{q-1}\chi(r)\log\left(\Gamma\left(\frac{r}{q}\right)\right) -\log(q) L(0,\chi) & \text{if} \  \chi  \ \text{is odd,}\\
			-\frac{1}{2} \Lambda(q)  & \text{if} \ \chi \ \text{is principle and} \ q>1,\\
			-\frac{1}{2} \log(2\pi)  & \text{if} \ q=1.
		\end{cases}
	\end{equation*}
\end{lemma}
For $ 0<x \leq 1$,   the Hurwitz zeta function is defined as 	 $\zeta(s,x):= \sum_{n =0}^\infty (n+x)^{-s}$.  The below lemma provides a relation between $L(s,  \chi)$ and $\zeta(s, x)$.  
	\begin{lemma}\cite[p.~249]{apostol} \label{Relation between L and Hurwitz zeta function}
	For  any character  $\chi $  modulo $q$, we have
	$$ 
	 L(s,\chi)= q^{-s} \sum_{r=1}^{q} \chi(r)\zeta\left(s,\frac{r}{q}\right).
	 $$	
\end{lemma}
\begin{lemma}\cite[p.~95]{Tit} \label{boundzeta}
	Suppose $s=\sigma + iR$ be a complex number. Then for any $\sigma\geq \sigma_0$, 
	$\exists$ a  
	constant $M(\sigma_0)$, such that 
	\begin{equation}
		|\zeta(s) | \ll |R|^{M(\sigma_0)} 
	\end{equation}
	as $|R| \rightarrow \infty.$
\end{lemma}

\begin{lemma} \cite[p.~97, Lemma 5.2] {IK} \label{bound_L(s,chi)}
	Let $s=\sigma + iR \in \mathbb{C}$ and $\chi$ be any character modulo $q$. Then for any  $\sigma_0 \leq \sigma \leq b$, $\exists$ a constant $A(\sigma_0)$, such that 
	\begin{equation}
		|L (s, \chi)| \ll |R|^{A(\sigma_0)}
	\end{equation}
as $|R|\rightarrow \infty.$
\end{lemma}

\begin{lemma} \label{laurent of zeta}\cite[p.~206]{Cohen-Henri} 
	The Laurent series expansion of $\zeta(s)$ around $s=1$ is given by
	\begin{equation}
		\zeta(s) = \frac{1}{s-1} + \sum_{k=0}^\infty (-1)^k \frac{\gamma_k}{k!}(s-1)^k = \frac{1}{s-1} +\gamma +O(s-1),
	\end{equation}
where the constants $\gamma_k$'s are called {\it  Stieltjes constants}.  Furthermore,  the Taylor series of $\zeta(s)$ around $s=0$ is given by 
\begin{equation}\zeta(s) =-\frac{1}{2} - \frac{1}{2}\log(2\pi) s + \frac{1}{48}(24\gamma_1 + 12\gamma^2 -\pi^2 -12(\log(2\pi))^2)s^2 + O(s^3).
\end{equation}
\end{lemma}

\subsection{Some trigonometric identities}
In this subsection,  we inscribe a few trigonometric identities which  will serve a major role in proving our main identities.
\begin{lemma}\cite[Lemma 4.3]{DGKM20}   \label{use in k even r even}
	Suppose $z \in \mathbb{C}$.  Then for any $m \in \mathbb{N}$,
	\begin{align}\label{sin/sin}
		\frac{\sin(mz)}{\sin(z)}=\sideset{}{''}\sum_{j=-(m-1)}^{(m-1)}  \exp(izj),
	\end{align}
	where $''$ indicates that the summation runs through $ j=-(m-1),-(m-3), \cdots ,(m-3),(m-1)$.
	Thus,  for $m$ even, 
	\begin{equation}\label{use in k even r odd}
		\frac{\sin(mz)}{\cos(z)}=(-1)^{\frac{m}{2}}\sideset{}{''}\sum_{j=-(m-1)}^{(m-1)}  i^j \exp (izj),
	\end{equation}
	and for $m$ odd,
	\begin{equation}\label{use in k odd r odd}
		\frac{\cos(mz)}{\cos(z)}=(-1)^{\frac{m-1}{2}}\sideset{}{''}\sum_{j=-(m-1)}^{(m-1)}   i^j \exp (-izj). 
	\end{equation}
\end{lemma}
\begin{lemma}\cite[Lemma 3.1]{DM20}\label{dixit_maji_lemma 1}
Suppose $\alpha,\beta,\gamma$ are three real numbers. Then we have
	\begin{align*}
		 2\Re\left(\frac{e^{i\alpha\beta}}{\exp(\gamma e^{-i\alpha})-1}\right)=\frac{\cos(\gamma\sin(\alpha)+\alpha \beta)-e^{-\gamma \cos(\alpha)}\cos(\alpha \beta)}{\cosh(\gamma\cos(\alpha))-\cos(\gamma\sin(\alpha))}.
	\end{align*}

\end{lemma}
	


The next section contains all the proofs of all the main identities of this paper.  

\section{Proof of the main identities}

Before going to the proof of the main Theorem \ref{Main theorem},  we first state the following lemma.
\begin{lemma} \label{integral lemma}
		Let $ x>0,$ $ N \in \mathbb{N}$ and $\chi$  be any Dirichlet character modulo $q$, $q \in \mathbb{N}$. Then for any $h \in \mathbb{Z}$, we have
		
		\begin{equation*}
					\sum_{r=1}^{q}\sum_{n=1}^{\infty}\frac{\chi(r)n^{N-2h}\exp{\left(-\frac{r}{q}n^N x\right)}}{1-\exp({-n^N x})} =\frac{1}{2\pi i} \int_{(c_0)} \Gamma(s) L(s,\chi) \zeta(Ns -(N -2h))  \left(\frac{q}{x}\right)^{s} \mathrm{d}s,
		\end{equation*}	
	where $c_0 = \Re(s) > \max \{1, (N-2h+1)/N\} .$ 
\end{lemma}
\begin{proof}
Note that for $x>0,  n \in \mathbb{N}, $ one has $|\exp{(-n^N x)}|<1 .$ We write
 \begin{equation} \label{exp(1)}
\frac{1}{1-\exp{(-n^N x)}} = \sum_{k=0}^{\infty}\exp{(-n^N x k)}.
	\end{equation}
Now, using inverse Mellin transform  \eqref{Inverse Mellin} of $\Gamma(s)$,  for $c_0>0$,  one has
\begin{align}
	\exp{ \left( -n^N x \left(k +\frac{r}{q}\right)\right)} &= \frac{1}{2 \pi i}\int_{(c_0)} \Gamma(s) \left( n^N x \left(k +\frac{r}{q}\right)\right)^{-s} \mathrm{d}s \nonumber \\
	&= \frac{1}{2 \pi i}\int_{(c_0)} \frac{\Gamma(s) x^{-s} \left(k +\frac{r}{q}\right)^{-s}}{n^{Ns}}  \mathrm{d}s.  \label{exp_interms_gamma}
\end{align}
Use \eqref{exp(1)} and \eqref{exp_interms_gamma}  to see that
	\begin{align*}
	\allowdisplaybreaks
	\sum_{r=1}^{q} \sum_{n=1}^{\infty} 
	\frac{\chi(r)n^{N-2h}\exp{\left(-\frac{r}{q}n^N x \right)}}{1-\exp({-n^N x})}
	&=  \sum_{r=1}^{q} \sum_{n=1}^{\infty} \chi(r) n^{N-2h} \sum_{k=0}^{\infty} \frac{1}{2 \pi i}\int_{(c_0)} \frac{\Gamma(s) x^{-s} \left(k +\frac{r}{q}\right)^{-s}}{n^{Ns}} \mathrm{d}s\\
	&= \frac{1}{2 \pi i} \int_{(c_0)}\Gamma(s) \sum_{r=1}^{q}\chi(r) \sum_{n=1}^{\infty}\frac{1}{n^{Ns-N +2h}}\sum_{k=0}^{\infty} \frac{1}{\left(k +\frac{r}{q}\right)^s}\frac{\mathrm{d}s}{x^{s}}  \\
	&= \frac{1}{2 \pi i} \int_{(c_0)} \Gamma(s) \zeta(Ns-(N-2h)) \sum_{r=1}^{q} \chi(r) \zeta\left(s, \frac{r}{q}\right) \frac{\mathrm{d}s}{x^{s}}  \\
	&= \frac{1}{2\pi i}\int_{(c_0)}\Gamma(s) \zeta(Ns-(N-2h)) L(s,\chi) \left(\frac{x}{q}\right)^{-s}\mathrm{d}s,
	\end{align*}
 for $c_0 > \max \{{1, {(N-2h+1)}/{N}}\}.$
  In the final step we have used Proposition \ref{Relation between L and Hurwitz zeta function}.
\end{proof}

\begin{proof}[Theorem \textup{\ref{Main theorem}}][]
From Lemma \ref{integral lemma},   for $\Re(s) = c_0 > \max\{ 1, (N-2h+1)/N\}$,   we have seen that
	\begin{equation} \label{integral}
	\sum_{r=1}^{q}\sum_{n=1}^{\infty}\frac{\chi(r)n^{N-2h}\exp{(-\frac{r}{q}n^N x)}}{1-\exp({-n^N x})} =\frac{1}{2\pi i} \int_{(c_0)}\Gamma(s)L(s,\chi)\zeta(Ns -(N-2h))  \left(\frac{x}{q}\right)^{-s} \mathrm{d}s.
\end{equation}
Now our main goal is to simplify the above integral and to do that we move the line integration  from $\Re(s) = c_0$ to $\Re(s) =d_0,$ where  we cleverly consider $d_0$ such that 
$$  - \left(2 \bigg\lfloor \frac{h}{N} \bigg\rfloor  +1 \right) - \epsilon < d_0  <  - \left(2 \bigg\lfloor \frac{h}{N}\bigg\rfloor  +1 \right)
$$ for some small positive $\epsilon$.   Later,  we shall explain the purpose of considering this bounds for $d_0$.  First,  we shall analyse the singularities of the integrand.  

Note that the only singularities of $\Gamma(s)$ are at non-positive integers and all of them are simple pole.   Again,  we know,  $s=1$ is the only simple pole of $\zeta(s)$.  Thus,  $s=(N-2h +1)/N$ is the only simple pole of $\zeta(Ns-(N-2h))$.   Moreover,  for any $ j\in \mathbb{N}$,  $s= (N-2h-2j)/N$ are the trivial zeros of $\zeta(Ns-(N-2h))$.   It is also given that $\chi$ is a primitive Dirichlet character modulo $q$.  For $q>1$,  one knows that $L(s,\chi)$ is an entire function and on the other hand, if $q=1$, it  has a simple pole at $s=1$.  
We know $\Gamma(s)$ has simple poles at non-positive integers.   We consider a contour $\mathcal{C}$ with the line segments $ [c_0 -iT, c_0 +iT], [c_0 +iT ,d_0 +iT], [d_0 +iT, d_0 -iT] $, and $ [d_0 -iT, c_0 -iT]$,  where $T$ is some large positive quantity.   
  Thus,  one can see that the only poles of integrand that are lying inside the contour $\mathcal{C}$,  namely,  at $s=0,  1,  (N-2h+1)/N,  -j$,  for $1 \leq j \leq 2 \lfloor \frac{ h}{N} \rfloor +1$,  and all are simple.

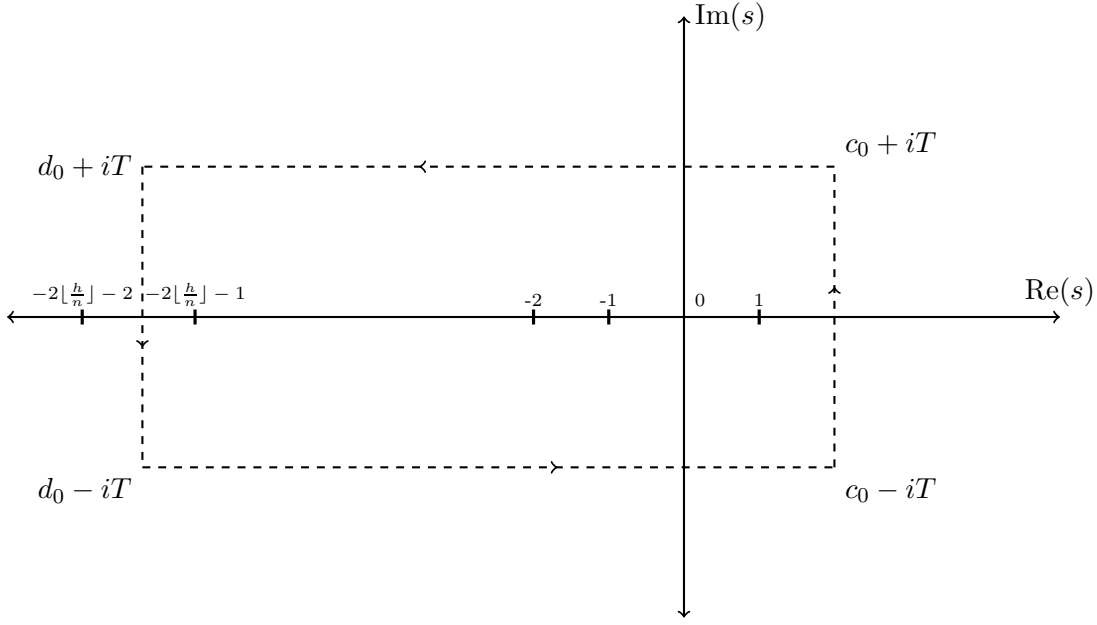
\begin{center}
	\begin{tikzpicture}[very thick,decoration={
			markings,
			mark=at position 0.6 with {\arrow{>}}}] 
		\draw[thick,dashed,postaction={decorate}] (-7.2,-2)--(2,-2) node[below right, black] {$c_{0}-i T$};
		\draw[thick,dashed,postaction={decorate}] (2,-2)--(2,2)  node[above right, black] {$c_{0}+iT$} ;
		\draw[thick,dashed,postaction={decorate}] (2,2)--(-7.2,2) node[left, black] {$d_{0}+i T$}; 
		\draw[thick,dashed,postaction={decorate}] (-7.2,2)--(-7.2,-2)  node[below left, black] {$d_{0}-i T$}; 
		\draw[thick, <->] (-9,0) -- (5,0) coordinate (xaxis);
		\draw[thick, <->] (-0,-4) -- (-0,4)node[midway, above right, black] {\tiny$0$} coordinate (yaxis);
		\draw (1,0.1)--(1,-0.1) node[midway, above, black] {\tiny1} ;
		\draw (-1,0.1)--(-1,-0.1) node[midway, above, black] {\tiny-1} ;
		\draw (-2,0.1)--(-2,-0.1) node[midway, above, black] {\tiny-2} ;
		\draw (-6.5,0.1)--(-6.5,-0.1) node[midway, above, black] {\tiny$- 2\lfloor \frac{h}{n} \rfloor-1$} ;
		\draw (-8,0.1)--(-8,-0.1) node[midway, above, black] {\tiny$-2\lfloor\frac{h}{n} \rfloor-2$} ;
		\node[above] at (xaxis) {$\Re(s)$};
		\node[right]  at (yaxis) {$\Im(s)$};
	\end{tikzpicture}
\end{center}
Now,  making use of  Cauchy's residue theorem,  we arrive at
 \begin{align} \label{residue theorem}
 	\frac{1}{2\pi i} \int_\mathcal{C} \Gamma(s) L(s,\chi) \zeta(Ns-(N-2h)) \left(\frac{q}{x}\right)^{s} \mathrm{d}s &   =   \mathcal{R}_0 +\mathcal{R}_1+ \mathcal{R}_{\frac{N-2h+1}{N}}  
 	 +  \sum_{j=1}^{2\lfloor\frac{h}{N} \rfloor+1} \mathcal{R}_{-j}, 
 \end{align}
where $\mathcal{R}_\alpha$  represents as the residual term at $s=\alpha$.  
 Letting $T \rightarrow \infty$ and employing Stirling's bound  \eqref{stirling equn} for $\Gamma(s)$  and together with Lemma \ref{boundzeta} and \ref{bound_L(s,chi)},  one can show that the horizontal integrals vanish.   Therefore,  from \eqref{residue theorem},  we get
 \begin{align}\label{4.3}
 		\frac{1}{2 \pi i} \left[ \int_{(c_0)} - \int_{(d_0)} \right] &  \Gamma(s) L(s,\chi) \zeta(Ns-(N-2h)) \left(\frac{q}{x}\right)^{s} \mathrm{d}s  
 		=  \mathcal{R}_0 + \mathcal{R}_1+ \mathcal{R}_{\frac{N-2h+1}{N}} + \sum_{j=1}^{2\lfloor\frac{h}{N} \rfloor+1} \mathcal{R}_{-j}. 
  \end{align}
Now employ \eqref{integral} in \eqref{4.3} to see that
\begin{align}\label{3.4}
		\sum_{r=1}^{q}\sum_{n=1}^{\infty}\frac{\chi(r)n^{N-2h}\exp{\left(-\frac{r}{q}n^N x\right)}}{1-\exp({-n^N x})}  = {V}_{N, h}(x; \chi)  + \sum_{j=1}^{2\lfloor\frac{h}{N} \rfloor+1} \mathcal{R}_{-j} \nonumber  \\
		 + \mathcal{R}_0 +   \mathcal{R}_1   +   \mathcal{R}_{\frac{N-2h+1}{N}},   
\end{align}
where 
\begin{equation}\label{left vertical integral}
	{V}_{N, h}(x; \chi) := \int_{(d_0)} \Gamma(s) L(s,\chi) \zeta(Ns -(N-2h)) \left(\frac{x}{q}\right)^{-s} \mathrm{d}s.
\end{equation}
The residue $\mathcal{R}_0$  is given by
\begin{align}\label{R_0}
	\mathcal{R}_0 &= \lim_{s\rightarrow 0} s \ \Gamma(s) \zeta(Ns-(N-2h)) L(s,\chi) \left(\frac{x}{q}\right)^{-s} = \zeta(-N+2h) L(0,\chi). 
\end{align}
Now we emphasis that the residue at $s=1$ depends on $q$.  Mainly,  the residue at $s=1$ is given by
\begin{align}\label{residue at 1}
 \mathcal{R}_1(x) :=	\mathcal{R}_1 
& = \begin{cases} \frac{\zeta(2h)}{x},  & \text{if} \ q=1,  \\
	0,  & \text{ if}\ q >1, 
\end{cases}
 \end{align}
where $\mathcal{R}_1(x)$ is the same function defined  as in \eqref{R(x)}.  The residue $\mathcal{R}_\frac{N-2h+1}{N}$ can be calculated by the following way
\begin{align}\label{R_N}
	\mathcal{R}_\frac{N-2h+1}{N} &= \lim_{s\rightarrow \frac{N-2h+1}{N}}\left(s- \frac{N-2h+1}{N}\right) \Gamma(s) \zeta(Ns-(N-2h)) L(s,\chi) \left(\frac{x}{q}\right)^{-s} \nonumber\\
	&= \frac{1}{N} \Gamma\left(\frac{N-2h+1}{N}\right) L\left( \frac{N-2h+1}{N},\chi \right) \left(\frac{x}{q}\right)^{-\left(\frac{N-2h+1}{N}\right)}. 
	\end{align}
Finally,  with the help of \eqref{Res_Gamaa} and \eqref{L-value},   the residue $\mathcal{R}_{-j}$ at $s=-j$,  for $j\in  \mathbb{N}$,  becomes
\begin{align}\label{R_{-j}}
	\mathcal{R}_{-j}  &= \lim_{s\rightarrow -j} (s+j) \Gamma(s) L(s,\chi)\zeta(Ns-(N-2h)) \left(\frac{x}{q}\right)^{-s} \nonumber\\
	&= \frac{(-1)^j}{j!} L(-j,\chi) \zeta(-Nj -N +2h)\left(\frac{x}{q}\right)^{j} \nonumber\\
	&= \frac{(-1)^{j+1}}{(j+1)!} B_{{j+1},\chi} \  \zeta(-Nj -N +2h)\left(\frac{x}{q}\right)^{j}.
\end{align}
Now the only thing is left is to show that the integral $ {V}_{N, h}(x; \chi)$ in \eqref{left vertical integral},  is nothing but the expression $\mathcal{J}_\chi(x)$,  defined as in \eqref{J}.
To simplify the integral \eqref{left vertical integral},   we first employ asymmetric form of the functional equations for $\zeta(s)$ and $L(s, \chi)$ respectively,  that is,  \eqref{asymetric functional equation for zeta} and \eqref{Dirichlet L-function}.   Thus,  upon simplification,  we see that
\allowdisplaybreaks
 \begin{align} \label{I(x)}
 {V}_{N, h}(x; \chi)	 & = \frac{ \mathcal{G}( \chi)}{ \pi i^a} \frac{1}{2 \pi i} \smashoperator{\int_{(d_0)}} \Gamma(s) \left( \frac{2\pi}{q} \right)^s \sin\left(\frac{\pi}{2}(s+a)\right) \Gamma(1-s) L(1-s, \bar{\chi})  \nonumber\\
 	& \times \frac{(2\pi)^{Ns}}{ \pi (2\pi)^{N-2h}}  \sin \left(\frac{\pi}{2}(Ns-(N-2h))\right) \Gamma(1-Ns +N- 2h) \nonumber\\ 
 	& \times \zeta(1-Ns+N-2h) \left(\frac{x}{q}\right)^{-s} \mathrm{d}s  \nonumber\\
 	= &\left(\frac{1}{2\pi}\right)^{N-2h+1} \frac{\mathcal{G}(\chi)}{i^a} \frac{1}{2 \pi i} \smashoperator{\int_{(d_0)}}\frac{2\sin\left(\frac{\pi}{2}(s+a)\right)}{\sin(\pi s)} \sin\left(\frac{\pi}{2}(Ns-(N-2h))\right) \nonumber \\
 	& \times L(1-s,\bar{\chi}) \Gamma(1-Ns+N-2h) \zeta(1-Ns+N-2h) \left(\frac{(2\pi)^{N+1}}{x}\right)^s \mathrm{d}s \nonumber \\
 	= &\left(\frac{1}{2\pi}\right)^{N-2h+1} \frac{\mathcal{G}(\chi)}{i^a} \frac{1}{2 \pi i} \smashoperator{\int_{(d_0)}} \left[\frac{\cos(\frac{\pi}{2}a)}{\cos(\frac{\pi}{2}s)} +\frac{\sin(\frac{\pi}{2}a)}{\sin(\frac{\pi}{2}s)} \right]\sin\left(\frac{\pi}{2}(Ns-(N-2h))\right)  \nonumber\\
 	& \times  L(1-s,\bar{\chi})\Gamma(1-Ns+N-2h) \zeta(1-Ns+N-2h) \left(\frac{(2\pi)^{N+1}}{x}\right)^s \mathrm{d}s.
 \end{align} 
Note that we have used the Euler's reflection identity \eqref{Euler's reflection} for $\Gamma(s)$ to obtain the second equality.  Now, we would like to change the variable $1-Ns+N-2h \longrightarrow s_1$. In that case, $d_0 < -2 \lfloor \frac{ h}{N} \rfloor -1 \Rightarrow d_0 <  -2 \frac{h}{N}+1 $ leads to   $\label{d_0} \Re(s_1) =d_1=1+N-2h- Nd_0 >1.$ 
Therefore,  after substitution,  \eqref{I(x)} becomes 
\begin{align}
	{V}_{N, h}(x; \chi) 
	=&  \left(\frac{1}{2\pi}\right)^{N-2h+1} \frac{\mathcal{G}(\chi)}{i^a N} \frac{1}{2 \pi i} \int_{(d_1)} \left[\frac{\cos(\frac{\pi}{2}a)}{\cos\left(\frac{\pi}{2} + \frac{\pi (1-s_1 -2h)}{2N}\right)} 
	 + \frac{\sin(\frac{\pi}{2}a)}{\sin\left(\frac{\pi}{2} + \frac{\pi (1-s_1 -2h)}{2N}\right)} \right] \nonumber \\
	 & \times \sin\left(\frac{\pi}{2}(1-s_1)\right) L\left(\frac{s_1 +2h -1}{N}, \bar{\chi}\right) \Gamma(s_1) \zeta(s_1)  \left(\frac{(2\pi)^{N+1}}{x}\right)^{\frac{1-s_1 +N -2h}{N}} \mathrm{d}s_1  \nonumber \\ 
	=  &\frac{\mathcal{G}(\chi)}{i^a N} \left(\frac{2\pi}{x}\right)^{\frac{N-2h+1}{N}}  \frac{1}{2 \pi i} \int_{(d_1)} \left[\frac{\cos(\frac{\pi}{2}a)\cos(\frac{\pi}{2}s_1)}{\sin\left(\frac{\pi}{2} \left(\frac{s_1 +2h-1}{N}\right)\right)} 
	+ \frac{\sin(\frac{\pi}{2}a) \cos(\frac{\pi}{2}s_1)}{\cos\left(\frac{\pi}{2} \left(\frac{s_1 +2h-1}{N}\right)\right)} \right] \nonumber\\
	& \times L\left(\frac{s_1 +2h -1}{N}, \bar{\chi}\right) \Gamma(s_1) \zeta(s_1) (X_N)^{-\frac{s_1}{N}}  \mathrm{d}s_1,  \nonumber \\
 	= &  \frac{\mathcal{G}(\chi)}{i^a N} \left(\frac{2\pi}{x}\right)^{\frac{N-2h+1}{N}}\left[\cos\left(\frac{\pi}{2}a\right) \mathcal{U}(X_N) + \sin\left(\frac{\pi}{2}a\right)\mathcal{V}(X_N) \right],\label{i(x)}
 \end{align}
where $ X_N = (2 \pi)^{N+1}/{x}$,  and 
\begin{align}
	\mathcal{U}(X_N) := \frac{1}{2 \pi i} \int_{(d_1)} \frac{\cos\left(\frac{\pi}{2} s_1\right)}{\sin\left(\frac{\pi}{2} \left(\frac{s_1 +2h-1}{N}\right)\right)} \Gamma(s_1) \zeta(s_1)  L\left(\frac{s_1 +2h -1}{N}, \bar{\chi}\right) X_N^{-\frac{s_1}{N}} \mathrm{d}s_1,  \label{U(X_N)}  \\
	\mathcal{V}(X_N) := \frac{1}{2 \pi i} \int_{(d_1)} \frac{\cos\left(\frac{\pi}{2} s_1 \right)}{\cos\left(\frac{\pi}{2} \left(\frac{s_1 +2h-1}{N}\right)\right)} \Gamma(s_1) \zeta(s_1)  L\left(\frac{s_1 +2h -1}{N}, \bar{\chi}\right) X_N^{-\frac{s_1}{N}} \mathrm{d}s_1.  \label{V(X_N)}
\end{align}
Now, our main aim is to evaluate the integrals  $\mathcal{U}(X_N)$ and $\mathcal{V}(X_N)$.  First,  let us concentrate on the integral $\mathcal{U}(X_N)$. To evaluate this integral,  we would like to expand $\zeta(s_1) L\left(\frac{s_1 +2h -1}{N}, \bar{\chi}\right) $ into a Dirichlet series.  We can do this as we know $\Re(s_1) >1 $ and $\Re( \frac{s_1 +2h -1}{N}) = 1 +\Re(s) >1 $ as $s_1 =1-Ns+N -2h$.  Hence,  we can write
\begin{align} \label{zeta and L}
	  L\left(\frac{s_1 +2h -1}{N}, \bar{\chi}\right)  \zeta(s_1) & =  \sum_{n=1}^{\infty}\frac{\bar{\chi}(n)}{n^{\frac{s_1 +2h -1}{N}}}  \sum_{m=1}^{\infty} \frac{1}{m^{s_1}} 
	= \sum_{m=1}^{\infty} \ \sum_{n=1}^{\infty}\frac{ \bar{\chi}(n)}{n^{\frac{2h-1}{N}}}\ (m^N n)^{-\frac{s_1}{N}}. 
\end{align}
Further,  one can check that
\begin{equation} \label{cos}
\cos\left(\frac{\pi}{2}s_1\right) = (-1)^{h+1} \sin\left(N\left(\frac{\pi}{2}\frac{s_1 +2h -1}{N}\right)\right).
\end{equation} 
Now,  invoking Lemma \ref{use in k even r even},  and together with  \eqref{zeta and L} and \eqref{cos},  the integral $\mathcal{U}(X_N)$ in \eqref{U(X_N)} becomes
\allowdisplaybreaks
\begin{align}\label{U(x_N)}
	\mathcal{U}(X_N) & =(-1)^{h+1} \frac{1}{2 \pi i} \int_{(d_1)} \sideset{}{''}\sum_{j=-(N-1)}^{(N-1)} \exp\left(\frac{ij \pi(s_1 +2h -1)}{2N}\right) \Gamma(s_1) \nonumber\\
	  & \hspace{4cm}\times \sum_{m=1}^{\infty}\  \sum_{n=1}^{\infty}\frac{ \bar{\chi}(n)}{n^{\frac{2h-1}{N}}}\ (m^N n)^{-\frac{s_1}{N}}X_N^{-\frac{s_1}{N}}   \mathrm{d}s_1 \nonumber\\
	 &= (-1)^{h+1} \sideset{}{''}\sum_{j=-(N-1)}^{(N-1)} \exp\left(\frac{i \pi j}{2N}(2h-1)\right) \sum_{m=1}^{\infty}\  \sum_{n=1}^{\infty}\frac{ \bar{\chi}(n)}{n^{\frac{2h-1}{N}}} \nonumber\\
	 & \hspace{4cm}\times   \frac{1}{2 \pi i} \int_{(d_1)} \Gamma(s_1) \left(X_N^{\frac{1}{N}}m n^{\frac{1}{N}}\exp{\left(-\frac{i \pi j}{2N}\right)}\right)^{-s_1} \mathrm{d}s_1.
\end{align}
Here,  we can verify that $\Re\left(X_N^{\frac{1}{N}}m n^{\frac{1}{N}}\exp{\left(-\frac{i \pi j}{2N}\right)}\right) =X_N^{\frac{1}{N}}m n^{\frac{1}{N}} \cos\left(\frac{ \pi j}{2N}\right)>0,$ since $X_N^{\frac{1}{N}}m n^{\frac{1}{N}}$
is a positive real number and $\cos\left(\frac{ \pi j}{2N}\right) >0$ as $j$ lies in the interval $-(N-1)\leq j \leq (N-1).$ Again,  using the inverse Mellin transform \eqref{Inverse Mellin} of $\Gamma(s)$ in \eqref{U(x_N)},  it reduces to
\begin{align}
&	\mathcal{U}(X_N) \nonumber \\
 =& (-1)^{h+1} \sideset{}{''}\sum_{j=-(N-1)}^{(N-1)} \exp\left(\frac{ i \pi j}{2N}(2h-1)\right)  \sum_{m=1}^{\infty}\  \sum_{n=1}^{\infty}\frac{ \bar{\chi}(n)}{n^{\frac{2h-1}{N}}} \exp\left(-X_N^{\frac{1}{N}}m n^{\frac{1}{N}}\exp{\left(-\frac{i \pi j}{2N}\right)}\right).   \label{2nd final form_U_X_N}
\end{align}
Now we  simplify the inner double sum as 
\begin{align}
 \sum_{n=1}^{\infty}\frac{\bar{\chi}(n)}{n^{\frac{2h-1}{N}}}\sum_{m=1}^{\infty} \left(\exp\left(-X_N^{\frac{1}{N}} n^{\frac{1}{N}}\exp{\left(-\frac{i \pi j}{2N}\right)}\right)\right)^m 
&= \sum_{n=1}^{\infty}\frac{ \bar{\chi}(n)}{n^{\frac{2h-1}{N}}} \frac{1}{\exp\left(X_N^{\frac{1}{N}} n^{\frac{1}{N}}\exp{\left(-\frac{i \pi j}{2N}\right)}\right)-1}.  \label{simple_sum}
\end{align}
Substitute \eqref{simple_sum} in \eqref{2nd final form_U_X_N} to obtain the final expression for $\mathcal{U}(X_N)$ as 
\begin{align}\label{u(x)}
	\begin{split}
	\mathcal{U}(X_N) = &(-1)^{h+1} \sideset{}{''}\sum_{j=-(N-1)}^{(N-1)} \exp\left(\frac{i \pi j}{2N}(2h-1)\right) \sum_{n=1}^{\infty}\frac{ \bar{\chi}(n)}{n^{\frac{2h-1}{N}}} \frac{1}{\exp\left(X_N^{\frac{1}{N}} n^{\frac{1}{N}}\exp{\left(-\frac{i \pi j}{2N}\right)}\right)-1}.  
		\end{split}
\end{align}
Now we shall try to simplify the integral $\mathcal{V}(X_N)$ defined in \eqref{V(X_N)}.  Moreover,  using the trigonometric identity \eqref{use in k even r odd},  and putting \eqref{zeta and L} and \eqref{cos} in \eqref{V(X_N)}, for $N \in 2 \mathbb{N}$,  we obtain
\begin{align*}
	\mathcal{V}(X_N) = &(-1)^{h+1}\frac{1}{2 \pi i}\int_{(d_1)} (-1)^{\frac{N}{2}}\sideset{}{''}\sum_{j=-(N-1)}^{(N-1)}\, i^j \exp\left(\frac{ij \pi(s_1 +2h -1)}{2N}\right) \Gamma(s_1)\\
	 & \hspace{5cm}\times \sum_{m=1}^{\infty}\  \sum_{n=1}^{\infty}\frac{ \bar{\chi}(n)}{n^{\frac{2h-1}{N}}}\ (m^N n)^{-\frac{s_1}{N}}X_N^{-\frac{s_1}{N}}  \mathrm{d}s_1.  
\end{align*}
Simplification of $ \mathcal{V}(X_N)$ goes in a similar direction to $\mathcal{U}(X_N)$,  hence we omit.  Upon simplification,  one can deduce that
\begin{align} \label{V_(X_N)}
	\mathcal{V}(X_N)  = (-1)^{h+1} & \sideset{}{''}\sum_{j=-(N-1)}^{(N-1)} (-1)^{\frac{N}{2}} i^j   \exp\left(\frac{ i \pi j}{2N}(2h-1)\right)  \sum_{n=1}^{\infty}\frac{ \bar{\chi}(n)}{n^{\frac{2h-1}{N}}} \frac{1}{\exp\left(X_N^{\frac{1}{N}} n^{\frac{1}{N}}\exp{\left(-\frac{i \pi j}{2N}\right)}\right)-1}.  
\end{align}
Note that,  the above expression for $\mathcal{V}(X_N)$ in \eqref{V_(X_N)} is true only for even $N$ as we have the identity \eqref{use in k even r odd},  which is valid for even $N$.  Now,  using \eqref{a} in \eqref{i(x)},  one can see that 
\begin{equation} \label{ij(x)}
{V}_{N, h}(x; \chi)  =\begin{cases}
		\frac{\mathcal{G}(\chi)}{N} \left(\frac{2\pi}{x}\right)^{\frac{N-2h+1}{N}} \mathcal{U}(X_N),   & \text { if } \chi \,  \text {is even},  \\
	\frac{1}{i}	\frac{\mathcal{G}(\chi)}{ N} \left(\frac{2\pi}{x}\right)^{\frac{N-2h+1}{N}} \mathcal{V}(X_N), & \text { if } \chi \,  \text {is odd}.
	\end{cases}
\end{equation}
Eventually, combining  \eqref{u(x)}-\eqref{ij(x)},  we arrive at
\begin{align}
{V}_{N, h}(x; \chi) =&(-1)^{h+1}\frac{\mathcal{G}(\chi)}{ N}\left(\frac{2\pi}{x}\right)^\frac{N-2h+1}{N} \sideset{}{''}\sum_{j=-(N-1)}^{N-1} v_{N,\chi}(j) \exp\left(\frac{i\pi j(2h-1)}{2N}\right) \nonumber
\\ &\times  G_{j}\left( \frac{ 2h-1}{N},  \frac{2\pi}{x},  \bar{\chi} \right),  \label{V_N,h}
\end{align}
where $v_{n,\chi}(j)$ and $G_{j}\left( \frac{ 2h-1}{N}, x,  \chi \right)$ is same as we defined in \eqref{vv} and \eqref{G-function}, respectively.  The above final expression of the vertical integral ${V}_{N, h}(x; \chi)$ is nothing but the expression $ \mathcal{J_\chi}(x)$ defined as in  \eqref{J}.  This settles the proof of Theorem \ref{Main theorem}. 
\end{proof}

\begin{proof}[Theorem {\textup{\ref{character analogue_GJKM}}}][]
Let $N \geq 1$ be an odd integer and $\chi$ be an even character modulo $q$.  
As Theorem \ref{Main theorem} is valid for any $h \in \mathbb{Z}$,  so we choose a particular sequence of $h$ such that $\frac{N-2h+1}{N}=-2m $ for any non-zero integer $m$.  
Further,   we substitute $x=2^{N}\alpha$ with $ \alpha \beta^N = \pi^{N+1}$  in \eqref{Main expression}  to see that 
$$ F(2Nm+1 ,  2^{N}\alpha ,  \chi) = \sum_{r=1}^{q}\sum_{n=1}^{\infty}\frac{\chi(r)\exp{\left(-\frac{r}{q}(2n)^{N} \alpha \right)}}{n^{2Nm+1} \left( 1 -\exp({-(2n)^N \alpha})\right)}, $$
and 
$$ G_{j}\left( 2m+1 , 2^{N} \alpha ,  \chi \right) =   \sum_{n=1}^{\infty}\frac{1}{n^{2m+1}} \frac{\chi(n)}{\exp \left( 4 \pi \left( n \alpha \right)^{1/N}\exp\left(-\frac{i\pi j}{2N}\right)\right)-1}.$$
Moreover,  we have
\allowdisplaybreaks
\begin{align}
		F(2Nm+1  ,  2^{N}\alpha , \chi)
	&	=  \zeta(2Nm+1) L(0,\chi) + \mathcal{R}_1(2^{N} \alpha)  + \frac{1}{N}\Gamma (-2m) L\left(-2m,\chi\right)\left(\frac{2^{N}\alpha}{q}\right)^{2m} \nonumber\\
	& +\sum_{j=1}^{2 \lfloor{\frac{N+1}{2N}+m}\rfloor+1} \frac{(-1)^{j+1}}{(j+1)!} B_{{j+1},\chi} \  \zeta(2Nm-Nj+1)\left(\frac{2^{N} \alpha }{q}\right)^{j}  + \mathcal{J}_\chi(2^{N} \alpha),  \label{DM}
	\end{align} 
where   \begin{equation} 
	\mathcal{R}_1(2^N \alpha):=\begin{cases}
		\frac{\zeta(N+1+2Nm)}{2^N \alpha},  & \text { if } q =1, \\
	0, & \text { if } q> 1, 
	\end{cases}\\
\end{equation}
and
\begin{align}\label{J}
	\mathcal{J}_\chi(2^{N} \alpha) = (-1)^{\frac{N+1}{2}+Nm+1}   \frac{\mathcal{G}(\chi)}{ N}  \left(\frac{2^{N} \alpha}{2\pi}\right)^{2m} \sideset{}{''}\sum_{j=-(N-1)}^{N-1} &  v_{N,\chi}(j) \exp\left(\frac{i\pi j(2m+1)}{2}\right) \nonumber \\
	 &\times G_{j}\left( 2m+1,  \frac{2\pi}{2^{N} \alpha},  \bar{\chi} \right).  
\end{align}
From \eqref{vv},  we know $v_{N,\chi}(j)=1$ as we are dealing with even character $\chi$.  Moreover,  from Lemma \ref{Zeros_Dirichlet_L} it is clear that $L(-2m,  \chi)=0$ and $\Gamma(s)$ has simple pole at $s=-2m$.  Letting $s=2m+1$ and $\chi$ even,  the functional equation \eqref{Dirichlet L-function} of $L(s, \chi)$ implies that 
\begin{equation}\label{use_function}
\Gamma(-2m) L(-2m,\chi) =	(-1)^m \frac{q}{2 \mathcal{G}(\overline{\chi}) } \left( \frac{q}{2 \pi} \right)^{2m}	L(2m+1,\overline{\chi}). 
	\end{equation}
Employ the relation $\alpha \beta^{N}=\pi^{N+1}$ to verify that
\begin{align}\label{some relation}
	\left(\frac{2^{N}\alpha}{2 \pi}\right)^{2m} & = 2^{2m(N-1)} \left( \frac{\alpha}{\beta} \right)^\frac{2Nm}{N+1}  \quad \text{and}\quad	2 \pi \left( \frac{2 \pi n}{2^{N} \alpha}\right)^\frac{1}{N}  =\beta  (2n)^\frac{1}{N}.  
	\end{align}
Making use of \eqref{use_function}, \eqref{some relation} and Euler's formula \eqref{zeta(2m)} in \eqref{DM},  we get
\allowdisplaybreaks\begin{align}
& \sum_{r=1}^{q}\sum_{n=1}^{\infty}\frac{\chi(r)\exp{\left(-\frac{r}{q}(2n)^{N} \alpha \right)}}{n^{2Nm+1} \left( 1 -\exp({-(2n)^N \alpha})\right)} - \zeta(2Nm+1) L(0,\chi) \\
&=    (-1)^m \frac{ 2^{2m(N-1)} q }{2N \mathcal{G}(\bar{\chi})}   \left( \frac{\alpha}{\beta} \right)^\frac{2Nm}{N+1} L(2m+1,  \bar{\chi}) \nonumber\\
	& + (-1)^{ m + \frac{N+3}{2}}   \frac{\mathcal{G}(\chi)}{ N} 2^{2m(N-1)} \left( \frac{\alpha}{\beta} \right)^\frac{2Nm}{N+1}  \nonumber \\
	& \times \sideset{}{''}\sum_{j=-(N-1)}^{N-1}  \exp\left(\frac{i\pi j(2m+1)}{2}\right) \sum_{n=1}^\infty \frac{\bar{\chi}(n)}{n^{2m+1} \left( \exp\left(  \beta (2n)^{1/N} e^{- \frac{i \pi j}{2N}}  \right)-1 \right)} + \mathcal{R}_1(2^N \alpha)  \nonumber \\
	 & +  (-1)^{ m + \frac{N+3}{2}}  2^{2Nm}  \sum_{j=1}^{\lfloor{\frac{N+1}{2N}+m}\rfloor}  \frac{(-1)^{j}}{q^{2j-1}} \frac{B_{{2j},\chi}}{ (2j)! } \frac{B_{2N(m-j)+N+1}}{(2N(m-j)+N+1 )!} \alpha^{2j+ \frac{2N(m-j)}{N+1}} \beta^{N+\frac{2N^2(m-j)}{N+1}} \nonumber \\
	 & + \mathcal{R}_1(2^N \alpha),  \label{Final_character}
\end{align}
where $\mathcal{R}_1(2^N \alpha)$ can be written as 
   \begin{equation} 
	\mathcal{R}_1(2^N \alpha):=\begin{cases}
	(-1)^{ m + \frac{N+3}{2}}  2^{2Nm} \frac{B_{2Nm+N+1}}{(2Nm+N+1)!}  \alpha^{\frac{2Nm}{N+1}} \beta^{N+\frac{2N^2 m}{N+1}}	,  & \text { if } q =1, \\
	0, & \text { if } q> 1. 
	\end{cases}\\
\end{equation}
Finally,  multiplying by $\alpha^{- \frac{2Nm}{N+1}}$ on both sides of \eqref{Final_character} and upon simplification,  one can complete the proof of Theorem \ref{character analogue_GJKM}. 

\end{proof}

\begin{proof}[Corollary \rm{\ref{Character analogue for q>1}}][]
Substituting $N=1$ with $q>1$ in Theorem \ref{character analogue_GJKM} and employing Lemma \ref{L(0,chi)},  we can immediately obtain \eqref{character_Katayama}. 
\end{proof}

\begin{proof}[Corollary \rm{\ref{L(1/3,chi)}}][]
	Putting $h=1$,  $N=3$ and $\chi= \chi_5 $,  an even primitive character modulo $5$ in Theorem \ref{Main theorem},   we get
	\begin{align} \label{equation 0}
		\begin{split}
		\sum_{r=1}^{5} \sum_{n=1}^{\infty} \chi_5(r) \frac{n \exp\left(-\frac{r}{5} n^3 x\right)}{1- \exp\left(-n^3 x\right)}& = \zeta(-1)L(0,\chi_5) + \frac{1}{3} \Gamma\left(\frac{2}{3}\right) L\left(\frac{2}{3}, \chi_5\right) \left(\frac{x}{5}\right)^{-\frac{2}{3}} \\
		&+ \frac{\mathcal{G}(\chi_5)}{3} \left(\frac{2\pi}{x}\right)^{\frac{2}{3}}
		 \sum_{n=1}^{\infty} \frac{\bar{\chi_5}(n)}{n^\frac{1}{3}}    \sideset{} {''}\sum_{j=-2}^{2} \frac{\exp\left(\frac{i\pi j }{6}\right)}{\exp \left(2\pi \left(\frac{2\pi n}{x}\right)^\frac{1}{3}\exp\left(\frac{-i\pi j}{6}\right)\right) -1}.
	\end{split}
	\end{align}
Note that $L(0,\chi_5)=0$ since $\chi_5$ is an even character.  
Now,  using the functional equation \eqref{Dirichlet L-function} of $L(s,\chi)$,  one can see that 
 \begin{equation} \label{L(2/3,chi)}
	L\left(\frac{2}{3},\chi_5\right)=\mathcal{G}(\chi_5)\left(\frac{2}{5}\right)^\frac{2}{3} \pi^{-\frac{1}{3}} \sin\left(\frac{\pi}{3}\right)\Gamma\left(\frac{1}{3}\right)L\left(\frac{1}{3},{\chi_5}\right).  
\end{equation}
Using \eqref{L(2/3,chi)}  and Euler's reflection formula \eqref{Euler's reflection},  we have
\begin{align}
	\label{equation 1}\frac{1}{3} \Gamma\left(\frac{2}{3}\right) L\left(\frac{2}{3}, \chi_5\right) \left(\frac{x}{5}\right)^{-\frac{2}{3}}  =   \frac{\mathcal{G}(\chi_5)}{3}  \left(\frac{2\pi}{x}\right)^{\frac{2}{3}} L\left(\frac{1}{3},\chi_5\right).
\end{align}
Upon simplification,  one can show that
\begin{align}
	\sideset{}{''}\sum_{j=-2}^{2} \frac{\exp\left(\frac{i\pi j }{6}\right)}{\exp \left(2\pi \left(\frac{2\pi n}{x}\right)^\frac{1}{3}\exp\left(\frac{-i\pi j}{6}\right)\right) -1} &= \frac{1}{\exp \left(2\pi \left(\frac{2\pi n}{x}\right)^\frac{1}{3}\right) -1} \nonumber \\
	&+ \frac{\exp\left(-\frac{i\pi }{3}\right)}{\exp \left(2\pi \left(\frac{2\pi n}{x}\right)^\frac{1}{3}\exp\left(\frac{i\pi }{3}\right)\right) -1}
\nonumber \\ &+\frac{\exp\left(\frac{i\pi }{3}\right)}{\exp \left(2\pi \left(\frac{2\pi n}{x}\right)^\frac{1}{3}\exp\left(-\frac{i\pi }{3}\right)\right) -1}.  \label{equation 2}
\end{align} 
It is clear that
{\allowdisplaybreaks \begin{align} 
& \left[\frac{\exp\left(-\frac{i\pi }{3}\right)}{\exp \left(2\pi \left(\frac{2\pi n}{x}\right)^\frac{1}{3}\exp\left(\frac{i\pi }{3}\right)\right) -1} +\frac{\exp\left(\frac{i\pi }{3}\right)}{\exp \left(2\pi \left(\frac{2\pi n}{x}\right)^\frac{1}{3}\exp\left(-\frac{i\pi }{3}\right)\right) -1}\right] \nonumber \\
&= 2\Re \left[\frac{\exp\left(\frac{i\pi }{3}\right)}{\exp \left(2\pi \left(\frac{2\pi n}{x}\right)^\frac{1}{3}\exp\left(-\frac{i\pi }{3}\right)\right) -1}\right].  \label{equation 3}
\end{align}}
Now, putting $ \gamma = 2\pi \left(\frac{2\pi n}{x}\right)^\frac{1}{3}$,  $\alpha =\frac{\pi}{3}$ and $\beta= 1$ in Lemma \ref{dixit_maji_lemma 1}, we get

\begin{align} \label{equation 4}
	2\Re \left(\frac{\exp\left(\frac{i\pi }{3}\right)}{\exp \left(2\pi \left(\frac{2\pi n}{x}\right)^\frac{1}{3}\exp\left(-\frac{i\pi }{3}\right)\right) -1}\right) = \frac{\cos\left(\sqrt{3}\pi\left(\frac{2\pi n}{x}\right)^\frac{1}{3}+\frac{\pi}{3}\right)-\frac{\exp\left(-\pi\left(\frac{2\pi n}{x}\right)^\frac{1}{3}\right)}{2}}{\cosh\left(\pi\left(\frac{2\pi n}{x}\right)^\frac{1}{3}\right)-\cos\left(\sqrt{3} \pi\left(\frac{2\pi n}{x}\right)^\frac{1}{3}\right)}. 
\end{align}
Again, 
\begin{align}\begin{split} \label{equation 5}
		\frac{1}{\exp \left(2\pi \left(\frac{2\pi n}{x}\right)^\frac{1}{3}\right) -1} &= \frac{\exp \left(-\pi \left(\frac{2\pi n}{x}\right)^\frac{1}{3}\right)}{\exp \left(\pi \left(\frac{2\pi n}{x}\right)^\frac{1}{3}\right) - \exp \left(-\pi \left(\frac{2\pi n}{x}\right)^\frac{1}{3}\right)}\\
		& =\frac{\exp \left(-\pi \left(\frac{2\pi n}{x}\right)^\frac{1}{3}\right)}{2\sinh\left(\pi \left(\frac{2\pi n}{x}\right)^\frac{1}{3}\right)}.
	\end{split}
\end{align}\\
Finally,  using \eqref{equation 1}-\eqref{equation 5} in \eqref{equation 0},  we obtain
\begin{align*}
	\sum_{r=1}^{5} \sum_{n=1}^{\infty} \chi_5(r) \frac{n \exp\left(-\frac{r}{5} n^3 x\right)}{1- \exp\left(-n^3 x\right)}&=\frac{\mathcal{G}(\chi_5)}{3}  \left(\frac{2\pi}{x}\right)^{\frac{2}{3}}\Bigg\{ L\left(\frac{1}{3},\chi_5\right) \\
	&+ \sum_{n=1}^{\infty} \frac{{\chi_5}(n)}{n^\frac{1}{3}}
 \left(\frac{e^{-u}}{2\sinh(u)} + \frac{\cos\left(\sqrt{3} u+\frac{\pi}{3}\right)-\frac{e^{-u}}{2}}{\cosh(u)-\cos\left(\sqrt{3} u\right)}\right)\Bigg\},
\end{align*}
Now substituting $x= \alpha^3$ and $ \beta = \frac{2 \pi^4}{\alpha^3}$,  one can finish the proof.   
\end{proof}

\begin{proof}[Theorem \rm{\ref{N-2h=-1 case}}][]
As  the proof goes in a similar direction as in Theorem \ref{Main theorem},  so
we only highlight the locations where the proof is different from Theorem \ref{Main theorem}. 
	We assumed that $N-2h= -1$,  thus Lemma \ref{integral lemma} gives,  for $\Re(s) = c_0 > 1$,
	\begin{equation}\label{Integrand N-2h=-1}
		\sum_{r=1}^{q} \sum_{n=1}^{\infty}\frac{\chi(r)}{n}\frac{\exp{(-\frac{r}{q}n^N x)}}{1-\exp({-n^N x})} =\frac{1}{2\pi i} \int_{(c_0)}\Gamma(s) L(s,\chi)\zeta(Ns +1) \left(\frac{x}{q}\right)^{-s} \mathrm{d}s.
	\end{equation}
 In this case,  the main difference is that $\zeta(Ns+1)$ has a pole at $s=0$, which we will divide three following cases.  
%
%
	To find residues,  we shall use the following Laurent series expansions around $s=0$:
	{\allowdisplaybreaks
	\begin{align} \label{Laurent series} 
	 \Gamma(s) &= \frac{1}{s} -\gamma + \frac{1}{2} \left(\gamma^2 + \frac{\pi^2}{6}\right)s +O(s^2), \nonumber\\
	 \zeta(s) &= \frac{1}{s-1} +\gamma+ O (s-1), \nonumber\\
	 \zeta(Ns+1)& = \frac{1}{Ns} +\gamma - \gamma_1 Ns + O(s^2),  \nonumber\\
		L(s,\chi)& = L(0,\chi) + L^{\prime}(0,\chi) s + O(s^2),  \nonumber\\
	  \left(\frac{x}{q}\right)^{-s} & = 1 - \log \left(\frac{x}{q}\right) s + O(s^2),  \nonumber\\
	  x^{-s}&  = 1 - \log (x) s + \frac{(\log (x))^2}{2} s^2 + O(s^3).
\end{align}}
The above Laurent series of $L(s,\chi)$ around $s=0$ is valid for non-principal primitive character modulo $q$ with $q>1$.	Now, we will discuss the cases one by one and calculate the residue at $s=0$ for each case.

	{\bf Case 1:} In this case,  we assume $\chi$ is  an  even  mod $q$ with $q=1$.  Then the integrand becomes $\Gamma(s)  \zeta(Ns+1) \zeta(s) x^{-s}$.
Note,  $s=0$ is a double pole of the integrand due to $\Gamma(s)$ and $\zeta(Ns+1)$. 
Using the definition of residue,  it will be 
	\begin{equation*}
	R_0 (x)	= \lim_{s\rightarrow 0} \frac{\rm d}{\mathrm{d}s}\left( s^2\Gamma(s) \zeta(Ns+1) \zeta(s)  x^{-s} \right). 
	\end{equation*} 
Using the above Laurent series expansions and upon simplification,  one can easily check that
\allowdisplaybreaks
\begin{align*}
 & \lim_{s\rightarrow 0}\left(\frac{\rm d}{\mathrm{d}s}\left( s^2\Gamma(s) \zeta(s) \zeta(Ns+1) x^{-s} \right)\right) = \frac{1}{2N} \left(\log(x) -\log (2\pi) -\gamma(N-1) \right). 
\end{align*}
Hence, 
\begin{equation}
R_0(x ) = \frac{1}{2N} \left(\log(x) -\log (2\pi) -\gamma(N-1) \right).
\end{equation}
 {\bf Case 2:} In this case, we consider $\chi$ as an even character modulo $q$
 with $q>1$. It is clear that $s=0$ is a simple zero of $L(s,\chi)$. 
 So, one pole at $s=0$ will get cancelled by this zero. 
 Thus,  $s=0$ will remain as a simple pole of integrand.  Hence, the residue at $s=0$ is given by 
 
 \begin{align*}
 	R_0(x) = \lim_{s\rightarrow 0 }s \Gamma(s) \zeta(Ns+1) L(s,\chi) \left({x}/{q}\right)^{-s}
 	&= \lim_{s\rightarrow 0} s\zeta(Ns+1) \frac{L(s,\chi)}{s} \\
 	=& \frac{1}{N}\lim_{s\rightarrow 0} \frac{L(s,\chi)}{s} \\
 	=& \frac{L^{\prime}(0,\chi)}{N}.
  \end{align*}
In the second last step, we have used the fact that $\zeta(Ns+1)$ has a simple pole at $s=0$ with residue $1/N$.  Therefore,  Lemma \ref{L'(0,chi)} gives
\begin{equation}
R_0(x)=	-\frac{1}{2N} \sum_{r=1}^{q-1} \chi(r) \log\left(
	\sin\left(\frac{r\pi}{q}\right)\right). 
\end{equation}
  Therefore, combining all these cases,  we can check that the residue at  $s=0$ is 
 	\begin{equation} \label{R_0(x)}
 	\mathcal{R}_0= 
 	\begin{cases}
 		\frac{1}{2N} \left(\gamma(1-N) -\log (2\pi) +\log(x)\right),  & \text{if} \ \chi \ \text{is even and} \  q=1,\\
 		-\frac{1}{2N} \sum_{r=1}^{q-1} \chi(r) \log\left(
 		\sin\left(\frac{r\pi}{q}\right)\right), & \text{if} \ \chi \ \text{is even and}  \ q>1,\\
 	\end{cases}\\
 \end{equation}
 which is exactly same as  we defined in \eqref{Residue at 0}. 
The residues at the remaining poles are same as in Theorem \ref{Main theorem}.
For example,  the residue at $s=1$,  that is,  $\mathcal{R}_1$ is 
\begin{equation} \label{r_1}
	\mathcal{R}_1 = \begin{cases}
	\frac{\zeta(N+1)}{x},  & \text{if} \ q=1, \\
	0, & \text{if} \ q>1, \\
	\end{cases}
\end{equation} 
 and the residue at $s=-j$ is given by
\begin{equation*}
R_{-j}(x)=	\frac{(-1)^{j+1}}{(j+1)!} B_{{j+1},\chi} \  \zeta(-Nj+1)\left(\frac{x}{q}\right)^{j},
\end{equation*}
where $1\leq j \leq 2 \lfloor \frac{N+1}{2N}\rfloor +1.$ 
Thus, the sum of these poles is given by
\begin{align} 
 \mathcal{R}(N) =\sum_{j=1}^{2\lfloor  \frac{N+1}{2N}\rfloor +1}R_{-j}(x) & =  \sum_{j=1}^{\lfloor \frac{N+1}{2N}\rfloor} \frac{B_{2j, \chi}}{(2j)!} \zeta(N+1-2Nj) \left(\frac{x}{q}\right)^{2j-1}  \nonumber \\
& = \begin{cases}
 \frac{x}{2q} L(-1,\chi), & \text{if} \ N=1, \\
 	0, & \text{if} \ N>1.
 \end{cases}\label{sum of r_-j}
\end{align}

The remaining proof is same as in  Theorem \ref{Main theorem}.  Thus,  in the proof of Theorem \ref{Main theorem},  we put $h=\frac{N+1}{2}$ in \eqref{V_N,h} to see that
\begin{align} \label{K_chi}
{V}_{N, h}(x; \chi) = (-1)^{\frac{N+3}{2}} \frac{\mathcal{G}(\chi)}{N} \sum_{n=1}^{\infty} \frac{\bar{\chi}(n)}{n} \sideset{}{''}\sum_{j=-(N-1)}^{N-1}  \frac{\exp\left(\frac{i \pi j}{2}\right)}{\exp\left(2 \pi \left(\frac{2\pi n}{x}\right)^{(1/N)} \exp{\left(\frac{-i \pi j}{2N}\right)}\right)-1}. 
\end{align}
Hence,  gathering all the residual terms  and  together with the above expression \eqref{K_chi},  we can conclude the proof of Theorem \ref{N-2h=-1 case}. 
\end{proof}

 \begin{proof}[Corollary \rm{\ref{for q=1}}][]
 Substituting $q=1$ in Theorem \ref{N-2h=-1 case},  one can immediately obtain this result. 
 \end{proof}
 \begin{proof}[Corollary \rm{\ref{for N=q=1}}][]
 Letting $N=q=1$ in Theorem \ref{N-2h=-1 case} and together with the fact that $\zeta(2)=\pi^2/6$ and $\zeta(-1)=-1/12$,  we can easily derive this identity. 
\end{proof}

\section{Concluding Remarks}

Inspired from the work of Kanemitsu et al.  \cite{KTY01},  Dixit and Maji studied the below infinite series,  for $N\in \mathbb{N},  h\in \mathbb{Z}$,  
\begin{align} \label{Dixit -Maji integral}
	\sum_{m=1}^{\infty} \frac{m^{N-2h}}{\exp(m^N x) - 1} 
\end{align}
Transformation formula for this series was crucial to find a new generalization for Ramanujan's identity for $\zeta(2m+1)$.  

 Later,   Kanemitsu et al.  further explored a character analogue of the series \eqref{Kanemitsu lambert series},  namely,  the following infinite series and its integral representation:
\begin{equation}\label{Kanemitsu_Dirichlet L}
	\sum_{r=1}^{q}\sum_{n=1}^{\infty}\frac{\chi(r)n^{N-2h}\exp{\left(-\frac{r}{q}n^N x\right)}}{1-\exp({-n^N x})} =\frac{1}{2\pi i} \int_{(c_0)}\Gamma(s) L(s,  \chi) \zeta(Ns -N +2h)  \left(\frac{x}{q}\right)^{-s} \mathrm{d}s,
\end{equation}
 where $\chi$ is a Dirichlet character modulo $q$,  and for some large positive $c_0$.   Although,  they studied it for $N\in \mathbb{N}$ and $h \in \mathbb{Z}$ with some restriction on $h$.  In this paper,  we studied the  same series \eqref{Kanemitsu_Dirichlet L} for any $N \in \mathbb{N}$ and $h\in \mathbb{Z}$.  This motivated us to find a new character analogue of Ramanujan's identity \eqref{Ramanujan's formula}.  For any $N\geq 1$ and $m \neq 0$,  our generalization will give a relation between $\zeta(2Nm+1)$ and $L(2m+1,  \chi)$.  
We were able to find a transformation formula for the series \eqref{Kanemitsu_Dirichlet L},  for any $N\in \mathbb{N}$ and $\chi$ is even,  whereas when $\chi$ is odd,  our transformation is valid only for $N\in 2\mathbb{N}$.  Thus,  it would be interesting to find a formula for odd $N \in \mathbb{N}$ when $\chi$ is an odd character.  
 It would also be an interesting problem to study the following line integral:
\begin{align}
 \frac{1}{2\pi i} \int_{(c_0)}\Gamma(s) L(s,  \chi) L(Ns -N +2h,  \psi)  \left(\frac{x}{QR}\right)^{-s} \mathrm{d}s,
\end{align}
where $\chi$ and $\psi$ are primitive characters modulo $Q$ and $R$,  respectively.

\section{Numerical verification of Theorem \ref{Main theorem}}
	\begin{table}[h]
		\centering
		{Let $N \in \mathbb{N}$ and $h\in \mathbb{Z}$ with $N-2h\neq-1$. Let $x$ be a positive real number. We took the left-hand side and right-hand side sum over $n$  considered only first $100$ terms.  This numerical data has been obtained using the Mathematica software. }
		\caption{Verification of Theorem \ref{Main theorem}} 
		\label{Table of main theorem}
		\renewcommand{\arraystretch}{1}
		{
			\begin{tabular}{|l|l|l|l|l|l|l|l|}
			
				\hline
				$N$ & $h$ & $x$ & $q$ & Dirichlet character $\chi$& Left-hand side  & Right-hand side  \\

				\hline
			  $	4$& $7$& $1.22$&$5$ & $\chi_{5,3}=(1,-1,-1,1)$ & $0.0929631$ &$0.0929631$\\      
				\hline
				$6$&$10 $&$\pi$ &$ 5$& $\chi_{5,2}= (1,  i,  -i, -1)$& $0.472922+0.138771i$ & $0.472921 +0.13877 i$  \\
				\hline
		$	8$	&$10$ &$\pi +1$ &$5$ &$\chi_{5,4}=(1,-i, i, -1)$ & $0.406854-0.109186 i$&$0.406807-0.109175 i$ \\
				\hline
			$10$	&$8$ & $e+1$& $7$& $\chi_{7,2}=(1,  e^{\frac{2\pi i}{3}}, e^{\frac{\pi i}{3}}, e^{-\frac{2\pi i}{3}}, e^{-\frac{\pi i}{3}},-1)  $&$0.462001 + 0.318763 i$ &$0.465286 + 0.318555 i$ \\
				\hline
				$2$& $-3$& $e$& $7$ &$\chi_{7,3}= (1,  e^{-\frac{2\pi i}{3}}, e^{\frac{2 \pi i}{3}}, e^{\frac{2\pi i}{3}}, e^{-\frac{2 \pi i}{3}},1)  $ &  $399.495-12.846 i$  & $399.495-12.846 i$ \\
				\hline
		\end{tabular}}
		
	\end{table}

\textbf{Acknowledgements.}
The last author wants to thank SERB for the MATRICS Grant MTR/2022/000545.  

\textbf{Data Availability Statement.} No data was used for the research described in the
article.

\end{document}